\documentclass{amsart}
\usepackage{amsmath}
\usepackage{amssymb}
\usepackage{amsthm}
\usepackage{graphicx}
\usepackage{tikz}

\newcommand{\inn}{\textrm{ in }}
\newcommand{\onn}{\textrm{ on }}
\newcommand{\nn}{\mathbf n}

\newcommand{\dimension}{d}
\newcommand{\lno}{\left\|}
\newcommand{\rno}{\right\|}
\newcommand{\lsn}{\left|}
\newcommand{\rsn}{\right|}
\newcommand{\alphab}{{\boldsymbol \alpha}}
\newcommand{\etab}{{\boldsymbol \eta}}
\newcommand{\xib}{{\boldsymbol \xi}}
\newcommand{\zetab}{{\boldsymbol \zeta}}
\newcommand{\xb}{{\mathbf x}}
\newcommand{\yb}{{\mathbf y}}
\newcommand{\Tb}{\mathbf T}
\newcommand{\Rb}{\mathbf R}
\newcommand{\EEE}{{\mathcal E}}
\newcommand{\KKK}{{\mathcal K}}
\newcommand{\GGG}{{\mathcal G}}
\newcommand{\HHH}{{\mathcal H}}
\newcommand{\RRR}{{\mathcal R}}
\newcommand{\JJJ}{{\mathcal J}}
\newcommand{\VVV}{{\mathbf V}}
\newcommand{\WWW}{{\mathbf W}}
\newcommand{\MMM}{{\mathcal M}}
\newcommand{\lp}{\left(}
\newcommand{\rp}{\right)}
\newcommand{\trino}{{|\hspace{-1pt} | \hspace{-1pt}|}}
\newtheorem{proposition}{Proposition}
\newtheorem{theorem}{Theorem}
\newtheorem{remark}{Remark}
\newtheorem{lemma}{Lemma}

\title[Method of Virtual Interfaces]{An Optimally Convergent Coupling Approach for Interface Problems Approximated with Higher--Order Finite Elements}
\author{James Cheung, Mauro Perego, Pavel Bochev, and  Max Gunzburger}

\begin{document}
\maketitle
%
%
\begin{abstract}
In this paper, we present a new numerical method for determining an optimally convergent numerical solution of interface problems on polytopial meshes. ``Extended'' interface conditions are enforced in the sense of a Dirichlet--Neumann coupling by means of a pullback onto the discrete interfaces. This coupling approach serves to bypass geometric variational crimes incurred by the classical finite element method. Further, the primary strength of this approach is that it  does not require that the discrete interfaces are geometrically matching to obtain optimal convergence rates. Our analysis indicates that this approach is well--posed and optimally convergent in $H^1$. Numerical experiments indicate that optimal broken $H^1$ and $L^2$ convergence is achieved. 
\end{abstract}
%
%
\section{Introduction}
Higher order finite element methods are attractive since they bring the prospect of faster converging numerical solutions for a lower computational cost. However, in many practical situations, higher order elements (i.e. elements with polynomial order of $2$ or greater) are not useful since the geometric approximation error of the polytopial mesh tends to dominate the best approximation error of the inherent polynomial approximation \cite[Chapter 4]{strang1973analysis}. As such, practical finite element computations are often performed using only piecewise linear or stabilized first order elements. Interface problems pose an additional difficulty since separate mesh approximation of the constituent subdomains may lead to geometrically nonmatching approximations to the interface. This commonly occurs when complicated domains must be meshed and also in cases where two different numerical codes must be merged together to compute the behavior of a coupled system, as it is often done for fluid--structure interaction problems. 

The most commonly utilized approach to overcome the issue of geometric non-coincidence is to incorporate \emph{transfer operators} to transfer values from one polytopial interface approximation to another \cite{de2007review}. These operators are used in instances of the Dirichlet--Neumann coupling method and mortar element methods \cite{flemisch2005new, flemisch2005mortar} for bridging together disjoint subdomain solutions. While these methods are simple and efficient, they suffer in the fact that the accuracy of their numerical solutions tend to be capped at second order in $L^2$ due to the geometric errors described in the previous paragraph. 

As in the case of simple boundary value problems, curvilinear maps can be used to better fit the discrete interface approximation to the interface given by the continuous problem. In \cite{bfer1985isoparametric} the isoparametric finite element method was generalized to the interface problem setting. Additionally, in \cite{bazilevs2006isogeometric}, the isogeometric analysis was applied to arterial blood flow. While these methods can provide higher--order numerical solutions, they can be restrictive in terms of computational cost since, in both cases, higher order quadrature rules must be utilized since the basis functions are no longer simple polynomials. In addition to the additional computational expense, methods based on curvilinear mappings can be laborious to implement. 

A notable method presented in \cite{qiu2016high} utilizes a similar idea to what is presented in this paper. In the approach presented there, optimal convergence rates are achieved by applying the high order finite element method presented in \cite{cockburn2014solving} to the interface problem setting, where the main idea of the approach is to utilize line integrals to optimally transfer data from the continuous interface onto its discrete approximations. While this approach allows for higher order numerical approximations, its biggest challenge lies in the fact that this approach is for mixed formulations of elliptic problems. Another similar method described in \cite{kuberry2017optimization} utilizes an optimization based approach to couple the extensions numerical solutions together onto a common refinement mesh generated from the vertices of the nonmatching interface approximations. In this approach, second order accuracy has been observed with linear elements. 

The purpose of this paper is to present a new numerical method for computing higher order numerical solutions for interface problems for cases when the polytopial interface approximations are not necessarily geometrically matching. This is done by extending the polynomial extension finite element method \cite{PE-FEM} to this setting by enforcing that the extension of the numerical solution and its extended co--normal derivatives are approximately weakly continuous on the interface given by the continuous problem. Since the continuous interface does not coincide with the geometry of the discrete problem, this matching condition is enforced by means of a pullback, via auxiliary variables, onto the discrete interface approximations. Because this method is based on affine--equivalent finite element approaches, it's implementation is relatively simple and its computational expense is comparable to that of classical finite element methods. We are able to demonstrate stability and optimal broken $H^1$ convergence theoretically, and optimal broken $L^2$ and $H^1$ convergence through a numerical example. 

The stucture of the paper is as follows: In \S \ref{section: Preliminaries}, we discuss the preliminary material required for this work. In \S \ref{section: Problem Setting}, we describe the elliptic interface problem and our numerical method. In \S \ref{section: Analysis} we state and prove our well--posedness theorem and error estimates. In \S\ref{section: Numerical Illustration} we provide a numerical illustration to vindicate the results of our analysis. And finally, in \S\ref{section: Concluding Remarks} we provide concluding remarks.
%
%
\section{Preliminaries} \label{section: Preliminaries}
In this section, we will discuss the preliminary notions required for this paper.
\subsection{Geometric Notions}
Let $k=2,3, \ldots$ and $\Omega_i \subset \mathbb R^\dimension$, where $\dimension = 2,3$ and $i=1,2$, denote bounded open subdomains having a $C^{k+1}$-- smooth boundary $\Gamma_i$ with an associated unit outer normal vector field $\nn_i$. We will assume that $\Gamma_1$ and $\Gamma_2$ intersect such that $\Omega_1 \cap \Omega_2 = \emptyset$ and $\lsn \Gamma^c \rsn > 0$, where $\Gamma^c := \Gamma_1 \cap \Gamma_2$ is the interface between the two subdomains $\Omega_1$ and $\Omega_2$. We will then denote $\Gamma^0_i:= \Gamma_i \setminus \Gamma^c$. 

We now define our discrete geometry. We will define $\Omega_{h,i}$ as the discrete polytopial approximation of $\Omega_i$ that arises from meshing. For simplicity, we shall also let $\Omega_{h,i}$ be the mesh triangulation, where $\KKK_{h,i}$ denotes an arbitrary simplicial element that belongs to $\Omega_{h,i}$ with meshsize $h_i:= \max_{\KKK_{h,i} \subset \Omega_h} \textrm{diam } \lp\KKK_{h,i}\rp$. Additionally, we will denote $h = \max \left\{h_1, h_2 \right\}$. Further, we will denote $\Gamma_{h,i}:= \partial \Omega_{h,i}$ and $\Gamma^c_{h,i}$ and $\Gamma^0_{h,i}$ as the subsets of $\Gamma_{h,i}$ that approximate $\Gamma^c_{h,i}$ and $\Gamma^0_{h,i}$ respectively. Further, we denote $\nn_{h,i}$ as the outer unit normal vector field associated with $\Gamma_{h,i}$. We remark that in the setting we consider in this paper, $\Gamma^c_{h,1}$ does not necessarily coincide with $\Gamma^c_{h,2}$. We assume however that all vertices of $\Gamma^c_{h,i}$ belong also to $\Gamma^c$. We shall denote $\EEE^j_{h,i}$ as an edge of $\Gamma_{h,i}$, where $j$ is an index variable. Associated with $\EEE^j_{h,i}$ is the element $\KKK^j_{h,i}$, where $\EEE^j_{h,i}$ belongs on the edge of $\KKK^j_{h,i}$.

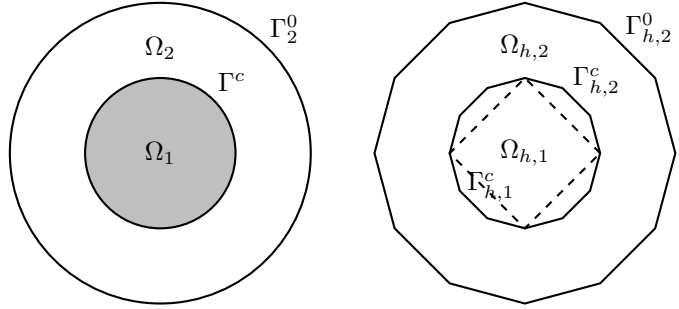
\begin{figure}[h!]
\begin{center}
\begin{tikzpicture}
	\filldraw[fill = white, thick] (0,0) circle(2);
	\filldraw[fill = lightgray, thick] (0,0) circle(1);
	\filldraw[black] (0,0) node[anchor=center]{$\Omega_1$};
	\filldraw[black] (0,1.4) node[anchor=center]{$\Omega_2$};
	\filldraw[black] (0.95,0.95) node[anchor=center]{$\Gamma^c$};
	\filldraw[black] (1.65, 1.65) node[anchor=center]{$\Gamma^0_2$};
\end{tikzpicture}
\qquad
\begin{tikzpicture}
\draw[black, thick] (0:2) \foreach \x in {30,60,...,359} {-- (\x:2)} -- cycle (90:1);
\draw[black, thick] (0:1) \foreach \x in {30,60,...,359} {-- (\x:1)} -- cycle (90:1);
\draw[black, dashed, thick] (0:1) \foreach \x in {0, 90, ..., 360} {--(\x:1)} -- cycle (90:1);
	\filldraw[black] (0,0) node[anchor=center]{$\Omega_{h,1}$};
	\filldraw[black] (0,1.4) node[anchor=center]{$\Omega_{h,2}$};
	\filldraw[black] (0.95,0.95) node[anchor=center]{$\Gamma^c_{h,2}$};
	\filldraw[black] (1.65, 1.65) node[anchor=center]{$\Gamma^0_{h,2}$};
	\filldraw[black] (-0.45, -0.45) node[anchor=center]{$\Gamma^c_{h,1}$};
\end{tikzpicture}
\end{center}
\caption{Left: Example of a continuous domain configuration. Right: Example of a discrete domain approximation.}
\end{figure}

Because we have assumed that $\Gamma_i$ is $C^{k+1}$--smooth, we have by the implicit function theorem that there exists, for every $\EEE^j_{h,i} \subset \Gamma_{h,i}$  a $C^{k+1}$ continuous mapping $\etab^j_i: \EEE^j_{h,i}\rightarrow \Gamma_{i}$. We can define these mappings such that $\etab^j_i: \EEE^j_{h,i}\rightarrow \Gamma_{h,i}$ such that $\bigcup_{j} \etab^j_i(\xib^j_i) = \Gamma^c$ and $\bigcap_j \etab^j(\xib^j_i) = \emptyset$, where $\xib^j_i$ is an arbitrary point of $\EEE^j_i$. Of course, since $\Gamma_{i}$ is sufficiently smooth, there exists a unique inverse for each $\etab^j_i$ for we will denote as $\zetab^j_i$. For the sake of notational convenience, we will denote $\etab_i: \Gamma_{h,i} \rightarrow \Gamma_i$ as the formal sum of all mappings $\etab^j_i$, and likewise $\zetab_i: \Gamma^c \rightarrow \Gamma^c_{h,i}$ as the formal sum of all $\zetab^j_i$. Additionally, we will denote $\xib_i$ as a point in $\Gamma_{h,i}$ and $\etab$ as a point in $\Gamma_i$. Further, we will define $\xib_i^c$ and $\xib_i^0$ as points in $\Gamma^c_{h,i}$ and $\Gamma^0_{h,i}$ respectively and $\etab^c$ as a point of $\Gamma^c$. Following this convention, we will then denote $\etab^c_i$ as the formal sum of $\etab^j_i$ that maps $\Gamma^c_{h,i}$ to $\Gamma^c$ and $\zetab^c$ as its inverse. Finally we will denote $\etab^0_i$ as the formal sum of $\etab^j_i$ that maps $\Gamma^0_{h,i}$ to $\Gamma^0_i$ and $\zetab^0_i$ as its inverse. 

We may now represent $\xib_1$ in terms of $\xib_2$ and vice versa by virtue of the invertibility of $\etab^j_i$. This idea is illustrated in the following schematic:
\begin{equation*}
\begin{aligned}
	\xib_{1,(2)}^c: \xib_1^c \overset{\etab^c_1}\rightarrow \etab^c \overset{\zetab^c_2}\rightarrow \xib^c_2\\
	\xib_{2, (1)}^c: \xib^c_2\overset{\etab^c_2}\rightarrow \etab^c \overset{\zetab^c_1}\rightarrow \xib^c_1,
\end{aligned}
\end{equation*}
where we have denoted $\xib^c_{1,(2)}:= \xib^c_1(\xib^c_2)$ and $\xib^c_{2, (1)}:=\xib^c_2(\xib^c_1)$. For simplicity, we will denote the pullback of a variable onto $\Gamma^c_{h,i}$ as $\mu^{(i)}$, e.g., if $\mu := \mu \lp \xib^c_1 \rp$ then $\mu^{(1)} := \mu\lp \xib^c_1 \rp$ and $\mu^{(2)}:= \mu \lp \xib^c_{1, (2)} \rp$.

As a final remark for this subsection, we have that since $\Gamma_{h,i}$ can be seen as a piecewise linear interpolant of $\Gamma_i$, we have that 
\begin{equation} \label{distance assumption}
	\lsn \etab_i(\xib_i) - \xib_i \rsn < \delta_{h,i} = \mathcal O(h_i^2),
\end{equation}
where $\delta_{h,i} \in \mathbb R^+$, and in addition, from a simple computation, we the following
\begin{proposition}
Let $\JJJ^c_{1, \lp 2 \rp}$ denote the Jacobian of the transformation $\Gamma^c_{h,1}\rightarrow \Gamma^c_{h,2}$, then the following bound is satisfied
\begin{equation} \label{eqn: Jacobian Bound}
	\lno \JJJ^c_{1,\lp2\rp} - 1\rno_{C^0\lp\overline{\Gamma^c_{h,2}} \rp} \leq C\lp h_1 + h_2 \rp \lno\etab_2 \rno_{C^2\lp\overline{\Gamma^c_{h,2}} \rp}.
\end{equation}
\end{proposition}
\subsection{Function Spaces and Discrete Lifting Operators}
Let $\alphab = (\alpha_i)_{i=1}^d$, $\alpha_i\ge0$ denote a multi-index, $|\alphab| = \sum_{i=1}^d \alpha_i$, and $\alphab! = \prod_{i=1}^d \alpha_i!$. For $\mathcal D=\Omega_i$ or $\Omega_{h,i}$ and for $m\in{\mathbb N}$, let $H^m(\mathcal D)$ denote the standard Sobolev space and $(H^m(\mathcal D))'$ the corresponding dual space; see \cite{adams2003sobolev}. For the purpose of studying the interface problems presented in this paper, we also need to consider the subspaces
\begin{equation*}
	H^1_{\Gamma^c_i}(\Omega_i) := \left\{ v \in H^1(\Omega_i) \,:\, \left. v\right|_{\Gamma^c_i} = 0 \right\}
\end{equation*}
and
\begin{equation*}
	H^1_{\Gamma^0_i}(\Omega_i) := \left\{ v \in H^1(\Omega_i) \,:\, \left. v\right|_{\Gamma^0_i} = 0 \right\}.
\end{equation*}

Also, for any $\xib\in\mathbb R^{\dimension}$, let $\xib^\alphab :=\xi_1^{\alpha_1}\xi_2^{\alpha_2}\cdots\xi_d^{\alpha_d}$ and $D^\alphab :=\partial^{|\alphab|}/\partial^{\alpha_1}\partial^{\alpha_1}\cdots\partial^{\alpha_d}$. For $\mathcal D=\Gamma^c$ or $\Gamma^c_{h,i}$, we consider the fractional Sobolev space $H^{m-\frac12}(\mathcal D)$. 
The $k$-th order Lagrange finite element space is defined by
\begin{equation*}
	V^k_{h,i} := \left\{ v\in C^0(\overline\Omega_{h,i}) \,\,:\,\, v|_{\KKK_{h,i}} \in P_k(\KKK_{h,i})\quad \forall \KKK_{h,i} \in \Omega_{h,i}\right\},
\end{equation*}
where $P_k(\KKK_{h,i})$ denotes the space of polynomials of order at most $k$ defined over a $d-$simplex $\KKK_{h,i}\in\mathbb R^\dimension$. 
and the trace spaces
\begin{subequations}
\begin{equation*}
	W^{c,k}_{h,i}:= V^k_{h,i}\big|_{\Gamma^c_{h,i}}= \left\{v \in C^0(\Gamma^c_{h,i})\,\,:\,\, v|_{\EEE^j_{h,i}} \in P_k(\EEE^j_{h,i}) \quad \forall\, \EEE^j_{h,i} \in \Gamma^c_{h,i}\right\}
\end{equation*}
\begin{equation*}
	W^{0,k}_{h,i}:= V^k_{h,i}\big|_{\Gamma^0_{h,i}}= \left\{v \in C^0(\Gamma^0_{h,i})\,\,:\,\, v|_{\EEE^j_{h,i}} \in P_k(\EEE^j_{h,i}) \quad \forall\, \EEE^j_{h,i} \in \Gamma^0_{h,i}\right\}
\end{equation*}
\end{subequations}
We also define the discontinuous finite element space
$$
\overline V^k_{h,i} := \left\{ v \in L^2(\Omega_{h,i}) :\,\, v|_{\KKK_{h,i}} \in P_k(\KKK_{h,i}) \quad \forall\, \KKK_{h,i} \in \Omega_{h,i} \right\}
$$
and the discrete differential operator $D_h^{\alpha} : \overline V^k_{h,i} \rightarrow L^2(\overline \Omega_{h,i})$ as follows: 
$$
D_h^{\alpha} v_h(\xb) := 
\begin{cases}
	D^\alpha v_h(\xib) &\textrm{ if } \xb \in \mathring{\KKK}_{h,i}\\
	0 &\textrm{ otherwise }.
\end{cases}
$$
Duality pairings over $\Omega_{h,i}$, $\Gamma^c_{h,i}$, and $\Gamma^0_{h,i}$ are defined by
\begin{equation*}
	\langle v, w \rangle_{\Omega_{h,i}} = \sum_{\KKK_{h,i}\in \Omega_{h,i}} \int_{\KKK_{h,i}} vw d\KKK_{h,i}\quad
\end{equation*}
\begin{equation*}
	\langle v, w \rangle_{\Gamma^c_{h,i}} = \sum_{j}\int_{\EEE^j_{h,i}} vw d{\EEE^j_{h,i}},
\end{equation*}
and
\begin{equation*}
	\langle v, w \rangle_{\Gamma^0_{h,i}} = \sum_{j} \int_{\EEE^j_{h,i}} vw d{\EEE^j_{h,i}},
\end{equation*}
respectively. 
``Broken'' Sobolev norms on $\Omega_{h,i}$, ${\Gamma^c_{h,i}}$, $\Gamma^0_(h,i)$ are defined by 
\begin{equation*}
\begin{aligned}
	\trino v \trino^2_{m, \Omega_{h,i}} &= \sum_{\KKK_{h,i}\in \Omega_{h,i}} \| v \|_{m, \KKK_{h,i}}^2 \,\,\,\,
	\forall\, v\in V^k_{h,i},\\
	\trino w \trino^2_{m, \Gamma^c_{h,i}} &= \sum_{\EEE^j_{h,i}\in \Gamma^c_{h,i}}\|w \|_{m, \EEE^j_{h,i}}^2
	\,\,\,\, \forall\, w\in W^{c,k}_{h,i} \quad\textrm{ and }\\
	\trino w \trino^2_{m, \Gamma^0_{h,i}} &= \sum_{\EEE^j_{h,i}\in \Gamma^0_{h,i}}\|w \|_{m, \EEE^j_{h,i}}^2
	\,\,\,\, \forall\, w\in W^{0,k}_{h,i},	
\end{aligned}
\end{equation*}
respectively. On the discrete spaces $V^k_{h,i}$, $W^{c,k}_{h,i}$, and $W^{0,k}_h$ we have the inverse inequalities involving the corresponding ``broken'' semi-norms given by
\begin{equation} \label{eqn: Inverse Inequality 1}
	\trino v \trino_{m, \Omega_{h,i}} \leq C h_i^{-1} \trino v \trino_{m-1, \Omega_{h,i}}
\,\,\,\,
	\forall, v\in V^k_{h,i},\; m=1,2,\ldots
\end{equation} 
\begin{equation*}
	\trino w \trino_{m+1/2, \Gamma^c_{h,i}} \leq Ch_i^{-\frac12} \trino w \trino_{m, \Gamma^c_{h,i}} \, \, \, \, \forall w\in W^{c,k}_{h,i},\; m=0,1,\ldots
\end{equation*}
and
\begin{equation*}
	\trino w \trino_{m+1/2, \Gamma^0_{h,i}} \leq Ch_i^{-\frac12} \trino w \trino_{m, \Gamma^0_{h,i}} \, \, \, \, \forall w\in W^{0,k}_{h,i},\; m=0,1,\ldots.
\end{equation*}

For simplicity of notation, we will define the product spaces
\begin{equation*}
	\mathbf H := H^1(\Omega_{h,1}) \times H^1(\Omega_{h,2}) \times H^{1/2}(\Gamma^c_{h,1}) \times H^{-1/2}(\Gamma^c_{h,2})
\end{equation*}
\begin{equation*}
	\VVV^k_h := V^k_{h,1}\times V^k_{h,2} \times W^{c,k}_{h,1} \times W^{c,k}_{h,2},
\end{equation*}
\begin{equation*}
	\WWW^k_h := W^{c,k}_{h,1} \times W^{c,k}_{h,2}
\end{equation*}
\begin{equation*}
	\mathbf H^{-1} := H^{-1}(\Omega_1) \times H^{-1}(\Omega_2),
\end{equation*}
and their norm taken to be the $\ell^2$ norm of their sub--norms. 

We will denote the $L^2\lp\Gamma^c_{h,i}\rp$ projection operator onto $W^{c,k}_{h,i}$ as $\pi^c_i\lp \cdot \rp: L^2\lp \Gamma^c_{h,i} \rp \rightarrow W^{c,k}_{h,i}$, defined by
\begin{equation}
	\int_{\Gamma^c_{h,i}} \mu \pi^c_2 w d\mathcal D= \int_{\Gamma^c_{h,i}} \mu w d\Gamma^c_{h,i} \quad  \forall \mu \in W^{c,k}_{h,i}.
\end{equation}

Lifting operators will be often used in this work. We will denote $\RRR^c_{h,i}: W^{c,k}_{h,i} \rightarrow V^k_{h,i} \cap H^1_{\Gamma^0_{h,i}}(\Omega_{h,i})$ and $\RRR^0_{h,i}: W^{0,k}_{h,i} \rightarrow V^k_{h,i} \cap H^1_{\Gamma^c_{h,i}}(\Omega_{h,i})$ as discrete bounded lifting operators. A simple inspection indicates that 
\begin{equation*}
	\textrm{Ker}\lp \RRR^c_{h,i}\rp \subset \textrm{Im}\lp \RRR^0_{h,i} \rp
	\quad\textrm{ and } \quad
	\textrm{Ker}\lp \RRR^0_{h,i}\rp \subset \textrm{Im}\lp \RRR^0_{h,i} \rp
\end{equation*}
For simplicity, we will denote $\RRR_{h,i} := \RRR^0_{h,i} + \RRR^c_{h,i}$. 

We conclude this subsection by establishing the following proposition, which is a simple consequence of the piecewise $C^{k+1}$--diffeomorphic equivalence property between $\Gamma^c, \Gamma^c_{h,1}$, and $\Gamma^c_{h,2}$.
\begin{proposition} \label{prop: trace norm equivalence}
There exists positive constants $c_1, c_2, C_1, C_2$ such that for $m \in \mathbb R$ the following norm equivalence relations are satisfied
\begin{equation*}
	c_1 \lno \mu_1 \rno_{m, \Gamma^c_{h,1}} \leq \lno \mu^{(2)}_1 \rno_{m, \Gamma^c_2} \leq C_1 \lno \mu_1 \rno_{m, \Gamma^c_{h,1}} \quad \forall \mu_1 \in H^m\lp \Gamma^c_{h,1} \rp
\end{equation*}
\begin{equation*}
	c_2 \lno \mu_2 \rno_{m, \Gamma^c_{h,2}} \leq \lno \mu_2^{(1)} \rno_{m, \Gamma^c_{h,1}} \leq C_2 \lno \mu_2 \rno_{m, \Gamma^c_{h,2}} \quad \forall \mu_2 \in H^m \lp \Gamma^c_{h,2} \rp.
\end{equation*}
\end{proposition}
\subsection{An Averaged Taylor Series Extension} \label{subsection: Averaged Taylor Series}
Recall that, for every $\EEE_{h,i}^j \subset \Gamma_{h,i}$, $\KKK_{h,i}^j$ is the element of $\Omega_{h,i}$ that contains $\EEE_{h,i}^j$. We will then let $\left\{ S_{h,i}^{j,j^\prime} \right\}$ be a family of disjoint star--shaped domains with respect to the balls $\sigma_{h,i}^{j,j^\prime} \subset \KKK_{h,i}^{j}$ such that $S_{h,i}^{j,j^\prime} \cap \KKK_{h,i}^{l} = \emptyset$ if $j \neq l$, $\textrm{diam}\lp S_{h,i}^{j,j^\prime}\rp = \mathcal O(\delta_{h,i})$ and $\overline{\bigcup_{j,j^\prime} S^{j, j^\prime}_{h,i}} \supset \overline{\Omega_i\Delta\Omega_{h,i}}$, where $\Omega_i\Delta\Omega_{h,i}:= \lp \Omega_i \cup \Omega_{h,i} \rp \setminus \lp \Omega_i \cap \Omega_{h,i} \rp$ denotes the symmetric difference of $\Omega_i$ and $\Omega_{h,i}$. We also require that $\overline{S^{j, j^\prime}_{h,i} \cap \etab_i\lp\EEE_{h,i}^j\rp} = \overline{S^{j, j^\prime}_{h,i}}\cap\etab_i\lp \EEE_{h,i}^j\rp$ and $\overline{S^{j,j^\prime}_{h,i}\cap\EEE_{h,i}^j} = \overline{S^{j, j^\prime}_{h,i}} \cap \EEE_{h,i}^j$. We refer to Figure \ref{kite} as an example of how star-shaped domains $S^{l,j}_{h,i}$ can be built for a triangular mesh.
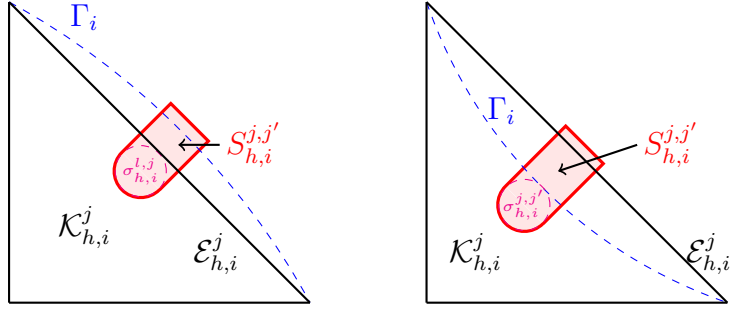
\begin{figure}[h!]
\begin{center}
\begin{tikzpicture}
\draw[red, very thick, fill= red!10!white] (1.5, 2) --(2.15, 2.65) -- (2.65, 2.15) -- (2.0,1.50)  (1.5, 2)  arc (135: 315: 0.3525);
\draw[blue, dashed] (0,4) .. controls (1.65, 3.1141) and (3.1141, 1.65) .. (4,0);
\draw[magenta, dashed] (1.75,1.75)  circle (0.34); 
\draw[black, thick] (0,0) -- (0,4);
\draw[black, thick] (0,0) -- (4,0);
\draw[black, thick] (0,4) -- (4,0);
\filldraw[black] (1.75, 1.75) node[anchor = center]{\color{magenta} \tiny $\sigma_{h,i}^{l,j}$};
\filldraw[black] (1, 1) node[anchor = center]{\color{black} \Large $\KKK_{h,i}^j$};
\filldraw[black] (2.75, 0.65) node[anchor = center]{\color{black} \Large $\EEE_{h,i}^j$};
\filldraw[black] (1,3.5) node[anchor = south]{\color{blue}\Large$\Gamma_i$};
\filldraw[black] (3.25, 2.1) node[anchor = center]{\color{red} \Large $S^{j,j^\prime}_{h,i}$};
\draw[black, thick, ->] (2.8, 2.1) -- (2.25, 2.1);
\end{tikzpicture}
\qquad \qquad 
\begin{tikzpicture}
\draw[red, very thick, fill= red!10!white] (1.5 - 0.45, 2- 0.45) --(2.15- 0.3, 2.65- 0.3) -- (2.65- 0.3, 2.15- 0.3) -- (2.0- 0.45,1.50- 0.45)  (1.5- 0.45, 2- 0.45)  arc (135: 315: 0.3525);
\draw[blue, dashed] (0,4) .. controls (1.65 - 1.0, 3.1141 - 1.0) and (3.1141- 1.0, 1.65 - 1.0) .. (4,0);
\draw[magenta, dashed] (1.75- 0.45,1.75- 0.45)  circle (0.34); 
\draw[black, thick] (0,0) -- (0,4);
\draw[black, thick] (0,0) -- (4,0);
\draw[black, thick] (0,4) -- (4,0);
\filldraw[black] (1.75 - 0.45, 1.75 - 0.45) node[anchor = center]{\color{magenta} \tiny $\sigma_{h,i}^{j,j^\prime}$};
\filldraw[black] (1 - 0.35, 1 - 0.35) node[anchor = center]{\color{black} \Large $\KKK_{h,i}^j$};
\filldraw[black] (2.75+ 1.0, 0.65) node[anchor = center]{\color{black} \Large $\EEE_{h,i}^j$};
\filldraw[black] (1,2.25) node[anchor = south]{\color{blue}\Large$\Gamma_i$};
\filldraw[black] (3.25, 2.1) node[anchor = center]{\color{red} \Large $S^{j,j^\prime}_{h,i}$};
\draw[black, thick, ->] (2.8, 2.1) -- (1.75, 1.75);
\end{tikzpicture}
\end{center}
\caption{Illustration of the construction of a star-shaped (with respect to $\sigma_{h,i}^{j,j^\prime}$) set $S_{h,i}^{j,j^\prime}\subset \mathbb R^\dimension$ for $\Gamma_i \cap \KKK_{h,i}^j = \emptyset$ (left) and  $\Gamma_i \cap \KKK_{h,i}^j \neq \emptyset$ (right).}\label{kite}
\end{figure}
Following \cite{brenner2007mathematical} we define, for $\xb \in \mathbb R^\dimension$ and $v \in L^2(\Omega_i \cap \Omega_{h,i})$, the averaged Taylor polynomial\footnote{The averaged Taylor polynomials on star-shaped domains $S^{j, j^\prime}_{h,i}$, are defined for functions in $L^1(\sigma^{j, j^\prime}_{h,i})$; see \cite[Corollary 4.1.15.]{brenner2007mathematical},}:
\begin{equation} \label{Taylor definition}
\left . T^{k}_{h,i} (v)\right |_{\xb} := \sum_{j,j^\prime} \mathbf 1_{S^{j,j^\prime}_{h,i}}(\xb) \displaystyle \int_{\sigma^{j,j^\prime}_{h,i}} {\left( \sum^k_{|\alphab| = 0} \frac{1}{\alphab !} D^\alphab v(\yb) (\xb - \yb)^\alphab  \phi_j(\yb) \right) d \yb }
\end{equation}
where $\phi_j(\yb)$ is a cutoff function with support over $\sigma^{j, j^\prime}_{h,i}$ and
\begin{equation*}
\mathbf 1_{S^{j, j^\prime}_{h,i}(\xb)} := 
	\begin{cases} 1 \quad \textrm{ if } \xb \in S^{j, j^\prime}_{h,i} \\ 
		0 \quad \textrm{ otherwise }
\end{cases}\end{equation*}
is the indicator function for the set $S^{j, j^\prime}_{h,i}$. Note that $T^k_{h,i}$ is meaningful only at $\xb \in \bigcup_{j,j^\prime} S^{j, j^\prime}_{h,i}$ and is zero otherwise. For any $\xib \in \Gamma_{h,i}$ and its image $\etab_i(\xib_i) \in \Gamma_i$ and $v \in L^2(\overline \Omega_i)$ we write 
\begin{equation} \label{Taylor on Gamma_h}
	v\circ\etab_i(\xib_i) = \left . T^k_{h,i} (v) \right |_{\etab_i(\xib_i)} + \left . R^k_{h,i} (v) \right|_{\etab_i(\xib_i)}.
\end{equation}
For $v \in H^{k+1}(\mathbb R^\dimension)$ we have $\|R^k_{h,i} (v) |_{\etab_i(\xib_i)}\|_{0, \Gamma^c_{h,i}} \le C \delta_{h,i}^{k+\frac12} |v|_{k+1, \mathbb R^\dimension}$ \cite[Appendix A, Lemma 2]{PE-FEM}
If $v \in \overline V^{k}_{h,i}$, then in every $\KKK_{h,i}^j$, $v$ is a polynomial of degree $k$, and therefore $T^k_{h,i} (v)$ reproduces exactly $v$ in any $\KKK_{h,i}^j$ adjacent to the boundary and it is equivalent to the classical Taylor polynomial. For $v\in \overline V^{k}_{h,i}$ we can therefore write, for a generic $\yb \in \KKK_{h,i}^j$:
$$
\begin{aligned}
\left . T^{k}_{h,i}(v) \right |_{\xb} &= \sum_{j, j^\prime} \mathbf 1_{S^{j, j^\prime}_{h,i}}(\xb) \sum^k_{|\alphab| = 0} \frac{1}{\alphab !} D^\alphab v(\yb)(\xb - \yb)^\alphab \\
&= \sum_j \mathbf 1_{\left(\cup _{j^\prime} S^{j, j^\prime}_{h,i}\right)}(\xb) \sum^k_{|\alphab| = 0} \frac{1}{\alphab !} D^\alphab v(\yb)(\xb - \yb)^\alphab
\end{aligned}
$$
We take now $\yb = \xib \in \mathring {\mathcal{E}}_i$ and $\xb = \etab_i(\xib_i)$, and we have that
\begin{equation} \label{Taylor discrete}
\left . T^{k}_{h,i}(v) \right |_{\etab_i(\xib_i)} = \sum^k_{|\alphab| = 0} \frac{1}{\alphab !} D_h^\alphab v(\xib_i)(\etab_i(\xib_i) - \xib_i)^\alphab
\end{equation}
which is well-defined for any $\xib_i \in \Gamma_{h,i}$ and $v\in \overline V_{h,i}^{k}$. For convenience, we also define $T^{k^\prime,k}_{h,i}$ as
\begin{equation} \label{Taylor discrete m-k}
\left . T^{k^\prime,k}_{h,i}(v) \right |_{\etab_i(\xib_i)} = \sum^k_{|\alphab| = k^\prime} \frac{1}{\alphab !} D_h^\alphab v(\xib_i)( \etab_i(\xib_i) - \xib_i)^\alphab.
\end{equation}
Clearly $T^{k}_{h,i} = T^{0,k}_{h,i}$. For vector functions $\mathbf v$, we introduce the vector operator $\mathbf{T}^{k}_{h,i}(\mathbf v) = \left(T^k_{h,i} v_i \right)_{i=1}^d$. We use this notation in particular for gradients of scalar functions (i.e., $\mathbf{T}^{k}_{h,i}(\nabla v)$).
\section{Problem Setting} \label{section: Problem Setting}
In this section we will describe the elliptic interface problem we wish to approximate the solution to and motivate the discrete coupling formulation we will use to accomplish this task.
\subsection{The Elliptic Interface Problem}
The continuous problem we consider in this paper is the elliptic interface problem with discontinuous coefficients given in the following
\begin{subequations} \label{eqn: Elliptic Interface Problem}
\begin{equation}
	\left\{
	\begin{aligned}
		-\nabla \cdot \lp p_i(\xb_i) \nabla u_i(\xb_i) \rp &= f_i \quad &\inn \Omega_i \\
										u_i &=0 \quad &\onn \Gamma_i^0
	\end{aligned}
	\right.
\end{equation}
\begin{equation} \label{eqn: Interface Conditions}
	\left.
	\begin{aligned}
		u_1 &= u_2 \\
		p_1(\xb_1)\frac{\partial u_1}{\partial \nn_1} + p_2(\xb_2)\frac{\partial u_2}{\partial \nn_2} &= 0
	\end{aligned}
	\right\} \onn \Gamma^c,
\end{equation}
\end{subequations}
where $\xb_i \in \Omega_i$. For the purpose of having $\lp u_1, u_2\rp$ belong to $H^{k+1}(\Omega_1)\times H^{k+1}(\Omega_2)$, as needed for determining optimal convergence rates for our method, we make the assumption that $p_i(\xb_i) \in C^{k+1}(\overline{\Omega_i})$ and $f_i(\xib_i) \in H^{k-1}(\Omega_i)$. As such, we have by the Sobolev extension theorem \cite[Chapter 5]{adams2003sobolev} that there exists an extension for each subdomain solution $\widetilde u_i \in H^{k+1}(\mathbb R^\dimension)$ such that $\left.\widetilde u_i\right|_{\Omega_i} = u_i$. Further, there exists extensions of $p_i$ and $f_i$ denoted as $\widetilde p_i$ and $\widetilde f_i$ respectively such that $\widetilde p_i \in C^{k+1}(\mathbb R^\dimension)$ and $\widetilde f_i \in H^{k-1}(\mathbb R^\dimension)$ and also $\left. \widetilde p_i \right|_{\Omega_i} = p_i$ and $\left. \widetilde f_i \right|_{\Omega_i} = f_i$. Further, these extensions are bounded in their respective norms, i.e.,
\begin{equation} \label{extension bounds}
\begin{aligned}
	\|\widetilde u_i\|_{k+1, \mathbb R^\dimension} &\leq C_e \|u_i\|_{k+1, \Omega_i} \\
	\|\widetilde p_i\|_{C^{k+1}(\mathbb R^\dimension)} &\leq C_e \|p_i\|_{C^{k+1}(\overline{\Omega_i})} \\
	\|\widetilde f_i\|_{k-1, \mathbb R^\dimension} &\leq C_e \|f_i\|_{k-1, \Omega_i}.
\end{aligned}
\end{equation}
\subsection{PE--FEM Subproblems Over $\Omega_{h,i}$} \label{PE-FEM Subproblems}
The Polynomial Extension Finite Element Method (PE-FEM) allows one to leverage the averaged Taylor series extensions described in Section \ref{subsection: Averaged Taylor Series} to obtain optimal convergence, with respect to interpolation, while retaining the generated polytopial mesh. Much of the details and explanations will be omitted for the sake of brevity. We refer the reader to \cite{PE-FEM} for more details behind the logic and intuition behind this approach. We will define two PE--FEM important problems that we will frequently refer to and state their stability properties here.

Let us define $a_{h,i}(\cdot, \cdot): H^1(\Omega_{h,i}) \times H^1(\Omega_{h,i}) \rightarrow \mathbb R$ as
\begin{equation*}
	a_{h,i}(w, v) := \int_{\Omega_{h,i}}\lp\widetilde p_i(\xb_i) \nabla w \cdot \nabla v\rp d\mathbf x\quad \forall w, v \in H^1(\Omega_{h,i}).
\end{equation*} 
Next, let us define $B^i_{h,D}(\cdot, \cdot): V^k_{h,i} \times V^k_{h,i}$ as
\begin{equation} \label{eqn: Dirichlet Bilinear Form}
	B^i_{h,D}(w, v) := a_{h,i}(w, v - \RRR_{h,i}v) + \theta_{h,i}\left< \left.T^k_{h,i} w\right|_{\etab_i\lp\xib_i\rp}, v \right>_{\Gamma_{h,i}} \quad \forall w,v \in V^k_{h,i},
\end{equation}
where $\theta_{h,i} \in \mathbb R^+$ is a positive constant such that $\theta_{h,i} \sim \mathcal O(h^{-1}_i)$ and $B^i_{h,N}(\cdot, \cdot): V^k_{h,i}\times V^k_{h,i}\rightarrow \mathbb R$ as
\begin{equation} \label{eqn: Neumann Bilinear Form}
\begin{aligned}
	B^i_{h,N}(w,v) &:= a_{h,i}(w, v - \RRR^0_{h,i}v) + \theta_{h,i} \left<\left.T^k_{h,i} w\right|_{\etab^0_i\lp\xib^0_i\rp}, v\right>_{\Gamma^0_{h,i}} \\
	&\qquad+ \tau_{i}(w,v) \quad\forall w,v \in V^k_{h,i} ,
\end{aligned}
\end{equation}
where $\tau_{i}: H^1(\Omega_{h,i}) \times H^{1/2}(\Gamma^c_{h,i})$ is defined as
\begin{equation} \label{eqn: Tau Definition}
	\tau_{i}(w, \mu) := \left< \widetilde p_i\circ\etab^c_i(\xib^c_i) \left.\Tb^{k-1}_{h,i}\nabla w\right|_{\etab^c_i\lp\xib^c_i\rp}\cdot \nn_i - \widetilde p_i(\xib^c_i)\nabla w \cdot \nn_{h,i}, \mu \right>_{\Gamma^c_{h,i}}
\end{equation}
for any $w \in H^1(\Omega_{h,i})$ and $\mu \in H^{1/2}(\Gamma^c_{h,i})$. 

We may now define the PE--FEM problems of interest for this paper. The importance of these problems will become apparent in the discussion of the coupling equations presented in \S \ref{sec: Coupling Equations}. First, consider the following: seek $u_{h,i} \in V^k_{h,i}$ such that 
\begin{equation} \label{eqn: Dirichlet Subproblem}
	B^i_{h,D}(u_{h,i}, v) = \left< \widetilde f_i, v - \RRR_{h,i} v \right>_{\Omega_{h,i}}
		+ \theta_{h,i} \left< g_D\circ\etab^c_i(\xib^c_i), v \right>_{\Gamma^c_{h,i}},
\end{equation}
where $g_D \in H^{1/2}(\Gamma^c)$. Second, consider: seek $u_{h,i}\in V^k_{h,i}$ such that
\begin{equation} \label{eqn: Neumann Subproblem}
	B^i_{h,N}(u_{h,i}, v) = \left< \widetilde f_i, v - \RRR^0_{h,i} v \right>_{\Omega_{h,i}}
		+ \left< g_N\circ\etab^c_i(\xib^c_i),v \right>_{\Gamma^c_{h,i}},
\end{equation}
where $g_N \in H^{-1/2}(\Gamma^c)$. The discrete problems \eqref{eqn: Dirichlet Subproblem} and \eqref{eqn: Neumann Subproblem} are meant to provide an optimal approximation to the solution of the following boundary value problems
\begin{equation*}
\left\{
\begin{aligned}
	-\nabla \cdot \lp p_i \nabla u_i \rp &= f_i \quad &\inn \Omega_i \\
						u_i &= 0 \quad &\onn \Gamma^0_i \\
						u_i &= g_D \quad &\onn \Gamma^c_i
\end{aligned}
\right.
\quad
\textrm{ and }
\quad
\left\{
\begin{aligned}
	-\nabla \cdot \lp p_i \nabla u_i \rp &= f_i \quad &\inn \Omega_i \\
						u_i &= 0 \quad &\onn \Gamma^0_i \\
						p_i \frac{\partial u_i}{\partial \nn_i} &= g_N \quad &\onn \Gamma^c_i
\end{aligned}
\right.
\end{equation*}
respectively. Following the spirit of the analysis presented in \cite[Appendix D]{PE-FEM} the following well--posedness results can be easily determined.
\begin{theorem}[Well--Posedness of \eqref{eqn: Dirichlet Subproblem}] \label{theorem: Dirichlet Subproblem Well-Posedness}
Let $B^i_{h,D}(\cdot,\cdot)$ be defined as in \eqref{eqn: Dirichlet Bilinear Form} with $\widetilde p_i(\xb_i) > 0$ everywhere in $\Omega_{h,i}$. Assume that $\theta_{h,i} \leq C_\theta h_i$ with $C_\theta$ chosen large enough and also assume that $\delta_{h,i} \sim \mathcal O(h_i^\frac32 )$. Then for $h_i$ small enough and $k = 1,2,\ldots$, we have that
\begin{equation} \label{eqn: Dirichlet Subproblem Boundedness}
	B^i_{h,D}(u,v) \leq M_{D,i} (1+\theta_{h,i}) \lno u \rno_{1, \Omega_{h,i}} \lno v \rno_{1, \Omega_{h,i}}
	\quad \forall u, v \in V^k_{h,i}
\end{equation}
and
\begin{equation} \label{eqn: Dirichlet Subproblem Coercivity}
	B^i_{h,D}(u,u) \geq \gamma_{D,i} \lno u \rno^2_{1, \Omega_{h,i}} \quad \forall u \in V^k_{h,i}.
\end{equation}
If $g_D\circ\etab^c_i(\xib^c_i) \in H^{1/2}(\Gamma^c_{h,i})$, then \eqref{eqn: Dirichlet Subproblem} has a unique solution and the solution satisfies the following stability bounds
\begin{subequations}
\begin{equation} \label{eqn: Dirichlet Subproblem Stability 1}
	\|u_{h,i}\|_{1, \Omega_{h,i}} \leq \Lambda_{D,i} \left\{ \lno \widetilde f_i \rno_{-1, \Omega_{h,i}} + \lno g_D\circ\etab^c_i(\xib^c_i) \rno_{1/2, \Gamma^c_{h,i}}  \right\}
\end{equation}
\begin{equation} \label{eqn: Dirichlet Subproblem Stability 2}
	\|u_{h,i}\|_{1, \Omega_{h,i}} \leq \Lambda_{D,i}^\prime \left\{ \lno \widetilde f_i \rno_{-1, \Omega_{h,i}} + h_i^{-\frac12}\lno g_D\circ\etab^c_i(\xib^c_i) \rno_{0, \Gamma^c_{h,i}}  \right\}.
\end{equation}
\end{subequations}
\end{theorem}
\begin{theorem}[Well--Posedness of \eqref{eqn: Neumann Subproblem}]
Let $B^i_{h,N}(\cdot, \cdot)$ be defined as in \eqref{eqn: Neumann Bilinear Form} with $\widetilde p_i(\xib_i) > 0$ everywhere in $\overline{\Omega_{h,i}}$. Then for $h_i$ small enough and $k=1,2,\ldots$ we have that
\begin{equation} \label{eqn: Neumann Subproblem Boundedness}
	B^i_{h,N}(u_{h,i},v) \leq M_{N,i}\lno u_{h,i} \rno_{1, \Omega_{h,i}} \lno v \rno_{1, \Omega_{h,i}}
		\quad \forall v \in V^k_{h,i},
\end{equation}
if $u_{h,i} \in V^k_{h,i}$ satisfies \eqref{eqn: Neumann Subproblem}, 
\begin{equation} \label{eqn: Neumann Subproblem Boundedness 2}
	B^i_{h,N}(u_{h,i}, \RRR^c_{h,i} \mu_i) \leq M_{N,i}^\prime\lno u_{h,i} \rno_{1, \Omega_{h,i}} \lno \mu_i \rno_{1/2, \Gamma^c_{h,i}}
		\quad \forall \mu_i \in W^{c,k}_{h,i},
\end{equation}
and 
\begin{equation} \label{eqn: Neumann Subproblem Coercivity}
	B^i_{h,N}(u,u) \geq \gamma_{N,i} \lno u\rno^2_{1, \Omega_{h,i}}
	\quad \forall u \in V^k_{h,i}.
\end{equation}
If $g_N\circ\etab^c_i(\xib^c_i) \in H^{-1/2}(\Gamma^c_{h,i})$, then \eqref{eqn: Neumann Subproblem} has a unique solution and the solution satisfies the following stability bound
\begin{equation} \label{eqn: Neumann Subproblem Stability}
	\|u_{h,i}\|_{1, \Omega_{h,i}} \leq \Lambda_{N,i} \left\{ \lno \widetilde f_i \rno_{-1, \Omega_{h,i}} + \lno g_N\circ\etab^c_i(\xib^c_i) \rno_{-1/2, \Gamma^c_{h,i}}      \right\}.
\end{equation}
\end{theorem}
\begin{remark}
	The constant $M_{N,i}$ in the continuity bound \eqref{eqn: Neumann Subproblem Boundedness} is independent of $\theta_{h,i}$ because we have enforced that the polynomial extension from $\Gamma^0_{h,i}$ onto $\Gamma^0_i$ is zero weakly, whereas the constant $M_{N,i}^\prime$ in the continuity bound \eqref{eqn: Neumann Subproblem Boundedness 2} is independent of $\theta_{h,i}$ because $\RRR^0_{h,i}\lp\RRR^c_{h,i} \mu_i\rp = 0$.
\end{remark}
\begin{remark}
	The stability constants $\Lambda_{D,i}, \Lambda_{D,i}^\prime$ and $\Lambda_{N,i}$ are independent of $\theta_{h,i}$.
\end{remark}
\subsection{The Coupling Formulation} \label{sec: Coupling Equations}
We will now present our coupling formulation. The comprehensive set of variational equations is to seek $\lp u_{h,2}, u_{h,2}, \lambda_h \rp \in \VVV^k_h$ that satisfies
\begin{equation} \label{eqn: MVI}
\begin{aligned}
\begin{aligned}
	B^1_{h,D}(u_{h,1}, v_1) - \theta_{h,1} \left<\lambda_h, v_1 \right>_{\Gamma^c_{h,1}}&= \left< \widetilde f_1, v_1 - \RRR_{h,1} v_1 \right>_{\Omega_{h,1}} \\
	B^2_{h,D}(u_{h,2}, v_2) - \theta_{h,2} \left<\lambda_h ^{(2)}, v_2 \right>_{\Gamma^c_{h,2}}  &= \left< \widetilde f_2, v_2 - \RRR_{h,2} v_2 \right>_{\Omega_{h,2}} \\
	B^1_{h,N}(u_{h,1}, \RRR^c_{h,1} \mu_1) + \left< \rho_h^{(1)}, \mu_1 \right>_{\Gamma^c_{h,1}} &= \left<\widetilde f_1, \RRR^c_{h,1} \mu_1 \right>_{\Omega_{h,1}}\\
	B^2_{h,N}(u_{h,2}, \RRR^c_{h,2} \mu_2) - \left< \rho_h , \mu_2 \right>_{\Gamma^c_{h,2}} &= \left<\widetilde f_2, \RRR^c_{h,2} \mu_2 \right>_{\Omega_{h,2}}
\end{aligned}
\\
\begin{aligned}
	\forall \lp v_1, v_2, \mu_1, \mu_2 \rp \in \VVV^k_h,
\end{aligned}
\end{aligned}
\end{equation}
From a simple inspection, it is easily determined that
\begin{equation*}
	\lambda_h = \Pi^0_{W^{c,k}_{h,1}}\lp \left.T^k_{h,1} u_{h,1}\right|_{\etab^c_1\lp \xi^c_1 \rp},\rp\end{equation*}
and the second equation of \eqref{eqn: MVI} implies that
\begin{equation} \label{eqn: Extension Dirichlet Matching Condition}
	\left. T^k_{h,2} u_{h,2} \right|_{\etab^c_2\lp \xib^c_{2} \rp} \approx \left.T^k_{h,1} u_{h,1}\right|_{\etab^c_2\lp \xib^c_{1,(2)} \rp}.
\end{equation}
and hence, the extension of the subdomain solutions match approximately on the continuous interface $\Gamma^c$.
Further inspection implies that
\begin{equation*}
	\left<\rho_h, \mu_2 \right>_{\Gamma^c_{h,2}} \approx \left<\widetilde p_2\circ\etab^c_2\lp \xib^c_2 \rp \left.\Tb^{k-1}_{h,2}\nabla u_{h,2}\right|_{\etab^c_2\lp \xib^c_2\rp} \cdot \nn_2, \mu_2 \right>_{\Gamma^c_h}
 \quad \forall \mu_2 \in W^{c,k}_{h,2}.
\end{equation*}
We refer the reader to \cite[Section 3]{PE-FEM} for a more detailed explanation. From this, it becomes apparent from the third equation of \eqref{eqn: MVI} that  
\begin{equation} \label{eqn: Extension Neumann Matching Condition}
	\sum_{i=1,2} \widetilde p_i\circ\etab^c\lp \xib^c \rp \left.\Tb^{k-1}_{h,i}\nabla u_{h,i}\right|_{\etab^c\lp \xib^c\rp} \cdot \nn_i  \approx 0.
\end{equation}
For narrative simplicity, we will refer to $\widetilde p_i\circ\etab^c\lp \xib^c \rp \left.\Tb^{k-1}_{h,i}\nabla u_{h,i}\right|_{\etab^c\lp \xib^c\rp} \cdot \nn_i$ as the \emph{extended co--normal derivative of $u_{h,i}$}. It then becomes clear from \eqref{eqn: Extension Dirichlet Matching Condition} and \eqref{eqn: Extension Neumann Matching Condition} that \eqref{eqn: MVI} approximates \eqref{eqn: Elliptic Interface Problem} by enforcing that the polynomial extensions of $u_{h,i}$ match weakly and that the extended co--normal derivatives are approximately balanced. Of course, the higher order convergence rates obtained by this method is due to the inclusion of the extension operators used to approximate the Dirichlet and Neumann interface conditions given by \eqref{eqn: Interface Conditions}. 
\section{Analysis} \label{section: Analysis}
In this section, we will present the necessary theoretical tools for the analysis of our discrete coupling formulation. We will then state and prove the well--posedness and $H^1$--optimality results.
\subsection{Dirichlet PE--FEM Solution Operators}
We introduce three operators for the purpose of the analysis of \eqref{eqn: MVI}. First, we define $\GGG_{h,i}(\cdot): H^{-1}(\Omega_{h,i}) \rightarrow V^k_{h,i}$ as the solution operator for the following problem: Given $\chi_i \in H^{-1}(\Omega_{h,i})$, seek $\phi_{h,i} \in V^k_{h,i}$ such that
\begin{equation} \label{eqn: Green's Function Dirichlet PE-FEM}
	B^i_{h,D}(\phi_{h,i}, v_i) = \left< \chi_i, v_i - \RRR_{h,i} v_i \right>_{\Omega_{h,i}} \quad \forall v_i \in V^k_{h,i}.
\end{equation}
We will denote $\GGG_{h,i} \chi_i := \phi_{h,i}$. Next, we define $\HHH^c_{h,i}(\cdot): H^{-1/2}(\Gamma^c_{h,i}) \rightarrow V^k_{h,i}$ as the solution operator of: seek $\psi^c_{h,i} \in V^k_{h,i}$ such that
\begin{equation} \label{eqn: Harmonic Dirichlet PE-FEM}
	B^i_{h,D}(\psi^c_{h,i}, v_i) = \theta_{h,i} \left< \nu, v_i \right>_{\Gamma^c_h}
	\quad \forall v_i \in V^k_{h,i}.
\end{equation}
We will similarly denote $\HHH^c_{h,i}\nu := \psi^c_{h,i}$, $\nu \in H^{-1/2}(\Gamma^c_{h,i})$. 
Finally, we will denote $\widehat\HHH^c_{h,i}: H^{-1/2}(\Gamma^c_{h,i}) \rightarrow V^k_{h,i} \cap H^1_{\Gamma^0_{h,i}(\Omega_{h,i})}$ as the solution operator of the variational problem: seek $\widehat \psi^c_{h,i} \in V^k_h$ such that
\begin{equation} \label{eqn: Perturbed Dirichlet PE-FEM}
	B^i_{h,D}(\widehat \psi^c_{h,i}, v_i) - \theta_{h,i}\left<\left.T^{1,k}_{h,i} \widehat\psi^c_{h,i}\right|_{\etab^0_{i}(\xib^0_i)}, v_i \right>_{\Gamma^0_{h,i}} = \theta_{h,i}\left< \mu, v\right>_{\Gamma^c_{h,i}} \quad \forall v_i \in V^k_{h,i}.
\end{equation}
Likewise, we denote $\widehat\HHH^c_{h,i}\nu := \widehat \psi^c_{h,i}$, $\nu \in H^{-1/2}(\Gamma^c_h)$. An inspection of \eqref{eqn: Perturbed Dirichlet PE-FEM} indicates that $\widehat \psi^c_{h,i} = 0$ on $\Gamma^c_{h,i}$ by means of the weak enforcement
\begin{equation*}
	\theta_{h,i}\left< \widehat \psi^c_{h,i}, v_i \right>_{\Gamma^0_{h,i}} = 0 \quad \forall v_i \in V^k_{h,i}.
\end{equation*}
\eqref{eqn: Perturbed Dirichlet PE-FEM} can easily be determined to be well--posed since the perturbation on $B^i_{D,h}(\cdot, \cdot)$ is miniscule. Through a modified set of steps presented in the analysis found in \cite[Appendix B]{PE-FEM}, we are able to demonstrate that the following stability bound is satisfied:
\begin{equation} \label{eqn: Perturbed Dirichlet PE-FEM Stability}
	\lno \widehat\HHH^c_{h,i} \nu \rno_{1, \Omega_{h,i}} \leq \widehat\Lambda_{D,i} \|\nu\|_{1/2, \Gamma^c_{h,i}}.
\end{equation}
This stability bound then allows us to prove the following.
\begin{lemma} \label{eqn: Perturbation Operator Difference}
	Let $\HHH^c_{h,i}\lp \cdot \rp$ and $\widehat\HHH^c_{h,i} \lp \cdot \rp$ be the solution operators for the problems defined in \eqref{eqn: Harmonic Dirichlet PE-FEM} and \eqref{eqn: Perturbed Dirichlet PE-FEM} respectively, then if $\delta \sim \mathcal O(h_i^2)$ we have that
\begin{equation}
	\lno \HHH^c_{h,i} - \widehat\HHH^c_{h,i} \rno_{H^{1/2}(\Gamma^c_{h,i}) \rightarrow H^1(\Omega_{h,i})} \leq Ch_i,
\end{equation}
\end{lemma}
\begin{proof}
We begin by seeing that $\lp\HHH^c_{h,i} - \widehat\HHH^c_{h,i}\rp (\cdot): H^{1/2}(\Gamma^c_{h,i}) \rightarrow V^k_{h,i}$ is the solution operator for the following: seek $\psi^c_{h,i} - \widehat\psi^c_{h,i}$ such that
\begin{equation*}
	B^i_{D,h}\lp \psi^c_{h,i} - \widehat \psi^c_{h,i}, v_i \rp 
	- \theta_{h,i}\left< \left.T^{1,k}_{h,i}\widehat\psi^c_{h,i}\right|_{\etab^0_i\lp\xib^0_i\rp}, v_i \right>_{\Gamma^0_{h,i}} = 0 \quad \forall v_i \in V^k_{h,i}
\end{equation*}
By taking the difference between \eqref{eqn: Harmonic Dirichlet PE-FEM} and \eqref{eqn: Perturbed Dirichlet PE-FEM}. Then \eqref{eqn: Dirichlet Subproblem Stability 2} of Theorem \ref{theorem: Dirichlet Subproblem Well-Posedness} implies that
\begin{equation*}
\begin{aligned}
	\lno\lp\HHH^c_{h,i} - \widehat\HHH^c_{h,i}\rp \nu \rno_{1, \Omega_{h,i}}
	&\leq \Lambda^\prime_{D,i}h_i^{-\frac12} \lno \left.T^{1,k}_{h,i}\lp \widehat\HHH^c_{h,i} \nu\rp\right|_{\etab^0_i\lp\xib^0_i\rp}\rno_{0, \Gamma^0_{h,i}} \\
	&\leq C\Lambda^\prime_{D,i}h_i \lno \widehat\HHH^c_{h,i}\nu \rno_{1, \Omega_{h,i}} \\
	&\leq C h_i \|\nu\|_{1/2, \Gamma^c_{h,i}},
\end{aligned}
\end{equation*}
after applying \cite[Appendix A, Lemma 4]{PE-FEM} and \eqref{eqn: Perturbed Dirichlet PE-FEM Stability}. The result of this lemma results from the definition of the operator norm.
\end{proof}
We now prove that $\HHH^c_h\lp\cdot \rp$ is invertible. This bound will be essential in the stability analysis.
\begin{lemma}
	Let $\HHH^c_{h,i}\lp \cdot \rp$ be defined as in \eqref{eqn: Harmonic Dirichlet PE-FEM}, then the following bound is satisfied
\begin{equation} \label{eqn: Harmonic Lower Bound}
	\lno \HHH^c_{h,i} \mu \rno_{1, \Omega_{h,i}} \geq C\lno \mu \rno_{1/2, \Gamma^c_{h,i}} \quad \forall \mu \in W^{c,k}_{h,i},
\end{equation}
under the assumption that $\delta_{h,i} \sim \mathcal O(h_i^2)$.
\end{lemma}
\begin{proof}
	Let $\psi^c_{h,i} = \HHH^c_{h,i} \mu$. Then we have that $\mu :=\pi^c_2\lp\left.T^{k}_{h,i} \psi^c_{h,i} \right|_{\etab^c_i\lp\xib^c_i\rp} \rp$. It then follows that
\begin{equation*}
\begin{aligned}
	\lno \psi^c_{h,i} \rno_{1, \Omega_{h,i}} &\geq C \lno \psi^c_{h,i} \rno_{1/2, \Gamma^c_{h,i}}\\
	& \geq C \lp \lno\mu\rno_{1/2, \Gamma^c_{h,i}} - h_i^{-\frac12}\lno\left. T^{1,k}_{h,i} \psi^c_{h,i} \right|_{\etab^c_i\lp\xib^c_i \rp} \rno_{0, \Gamma^c_{h,i}} \rp \\
	& \geq C \lp \lno \mu \rno_{1/2, \Gamma^c_{h,i}} - \sum_{|\alphab| = 1} \delta_{h,i}^{|\alphab|} h_i^{-|\alphab| + \frac12} \lno \psi^c_{h,i} \rno_{1, \Omega_{h,i}} \rp \\
	&\geq C \lno \mu \rno_{1/2, \Gamma^c_{h,i}},
\end{aligned}
\end{equation*}
after applying \cite[Appendix A, Lemma 4]{PE-FEM}, \eqref{eqn: Dirichlet Subproblem Stability 1}, and the assumption that $\delta_{h,i}\sim \mathcal O(h^2_i)$.
\end{proof}
\subsection{Stability Analysis}
We will now analyze the well--posedness of our coupling formulation. First, we begin by seeing that, by a change of variables for integrals, \eqref{eqn: MVI} is equivalent to seeking $\lp u_{h,1}, u_{h,2}, \lambda_h, \rho_h \rp \in \VVV^k_h$ such that
\begin{equation} \label{eqn: MVI Equivalent}
\begin{aligned}
\begin{aligned}
	B^1_{h,D}(u_{h,1}, v_1) - \theta_{h,1} \left<\lambda_h, v_1 \right>_{\Gamma^c_{h,1}}&= \left< \widetilde f_1, v_1 - \RRR_{h,1} v_1 \right>_{\Omega_{h,1}}  \\
	B^2_{h,D}(u_{h,2}, v_2) - \theta_{h,2} \left<\lambda_h^{(2)}, v_2 \right>_{\Gamma^c_{h,2}}  &= \left< \widetilde f_2, v_2 - \RRR_{h,2} v_2 \right>_{\Omega_{h,2}} \\
	\sum_{i=1,2} B^i_{h,N}\lp u_{h,i}, \RRR^c_{h,i} \mu_i \rp + \left< \rho_h, \JJJ^c_{1, (2)}\mu_1^{(2)} - \mu_2  \right>_{\Gamma^c_{h,2}} &= \sum_{i=1,2} \left<\widetilde f_i, \RRR^c_{h,i} \mu_i \right>_{\Omega_{h,i}}
\end{aligned}
\\
\begin{aligned}
	\forall \lp v_1, v_2, \mu_1, \mu_2 \rp \in \VVV^k_h.
\end{aligned}
\end{aligned}
\end{equation}
The first two equations in \eqref{eqn: MVI Equivalent} implies that
\begin{equation} \label{eqn: Solution Decomposition}
\begin{aligned}
	u_{h,i} &= \HHH^c_{h,1} \lambda_h^{{i}} + \GGG^c_{h,i} \widetilde f_i \\
\end{aligned}
\end{equation}
and thus the third equation in \eqref{eqn: MVI Equivalent} can be written as
\begin{equation} \label{eqn: Flux Coupling Equation}
\begin{aligned}
	&\sum_{i=1,2} B_{h,N}^i\lp \HHH^c_{h,i} \lambda_h^{(i)}, \RRR^c_{h,i} \mu_i \rp + \left< \rho_h, \JJJ^c_{1, (2)}\mu_1^{(2)}- \mu_2  \right>_{\Gamma^c_{h,2}} \\
	&\qquad= \sum_{i=1,2}\lp \left<\widetilde f_i, \RRR^c_{h,i} \mu_i \right>_{\Omega_{h,i}} - B_{h,N}^i\lp\GGG^c_{h,i}\widetilde f_i, \RRR^c_{h,i} \mu_i \rp \rp \quad \forall \lp \mu_1, \mu_2 \rp\in \WWW^k_h.
\end{aligned}
\end{equation}
This equation will be of paramount importance in the following analysis, as it allows us to determine a bound $\lambda_h$.
We now prove the following well--posedness result.
\begin{theorem} \label{theorem: Stability Bound}
Assume that $\widetilde p_i > 0$ everywhere on $\Omega_{h,i}$ and $\widetilde f_i \in H^{-1}(\Omega_{h,i})$ and $\theta_{h,i} \geq C_\theta h_i^{-1}$ with $C_\theta$ large enough. Then if $\delta_{h,i} \sim O(h_i^2)$ with $h_i$ small enough, the following stability bound is satisfied
\begin{equation}
	\lno \lp u_{h,1}, u_{h,2}, \lambda_h, \rho_h\rp \rno_{\mathbf H} \leq C\lno \lp\widetilde f_1, \widetilde f_2 \rp\rno_{\mathbf H^{-1}}.
\end{equation}
for $k=1,2,\ldots,$ and $\dimension = 2,3$.
\end{theorem}
\begin{remark}
	This stability result implies that the solution to \eqref{eqn: MVI Equivalent} is unique, since the equations are linear.
\end{remark}
\begin{proof}
First, we begin by recalling that
\begin{equation*}
	\left< \rho_h, \mu_2 \right>_{\Gamma^c_{h,2}} = B^2_{h,N}\lp u_{h,2}, \RRR^c_{h,2} \mu_2 \rp - \left< \widetilde f_2, \RRR^c_{h,2} \mu_2 \right>_{\Gamma^c_{h,2}}.
\end{equation*}	
By applying \eqref{eqn: Neumann Subproblem Boundedness}, the solution decomposition \eqref{eqn: Solution Decomposition}, and \eqref{eqn: Dirichlet Subproblem Stability 1}, we then have that
\begin{equation} \label{eqn: Rho Bound}
	\lno \rho_h \rno_{-1/2, \Gamma^c_{h,2}} \leq C\lp \lno \lambda_h \rno_{1/2, \Gamma^c_{h,1}} + \lno \widetilde f_2 \rno_{-1, \Omega_{h,2}} \rp.
\end{equation}

Next, we proceed by choosing $\RRR^c_{h,i}\lp\cdot\rp = \widehat\HHH^c_{h,i} \lp\cdot\rp$, $\mu_1 = \lambda_h$, $\mu_2 = \pi^c_2\lambda_h^{(2)}$. The third equation in \eqref{eqn: MVI Equivalent} then becomes
\begin{equation*}
\begin{aligned}
	\sum_{i=1,2} B^i_{h,N}\lp \HHH^c_{h,i}\lambda_h^{(i)}, \widehat\HHH^c_{h,i}\lambda_h^{(i)} \rp &= \sum_{i=1,2} \left<\widetilde f_i, \widehat\HHH^c_{h,i} \lambda_h^{(i)} \right>_{\Omega_{h,i}} - B^i_{h,N}\lp \GGG^c_{h,i}\widetilde f_i, \widehat \HHH^c_{h,i} \lambda_h^{(i)} \rp \\
	&- \left< \rho_h, \lp \JJJ^c_{1,(2)} - 1 \rp \lambda_h^{(2)} \right>_{\Gamma^c_{h,2}} 
\end{aligned}
\end{equation*}
where we have utilized the projection theorem to see that
\begin{equation*}
	\left< \nu, \lambda_h - \pi^c_2\lambda_h^{(2)} \right>_{\Gamma^c_{h,2}} = 0 \qquad \forall \nu \in W^{c,k}_{h,2},
\end{equation*}
and the identity $\widehat\HHH^c_{h,2}\lp \pi^c_2 \lambda_h^{(2)} \rp = \widehat\HHH^c_{h,2} \lambda_h^{(2)}$. Using \eqref{eqn: Jacobian Bound}, \eqref{eqn: Rho Bound}, and Proposition \ref{prop: trace norm equivalence},  we have that
\begin{equation*}
\begin{aligned}
	\left< \rho_h, \lp \JJJ^c_{1,(2)} - 1 \rp \lambda_h^{(2)} \right>_{\Gamma^c_{h,2}} 
	&\leq Ch \lno \rho_h \rno_{-1/2, \Gamma^c_{h,2}} \lno \lambda_h \rno_{1/2, \Gamma^c_{h,1}}  \\
	&\leq Ch\lp \lno \lambda \rno^2_{1/2, \Gamma^c_{h,1}} + \lno \widetilde f_1 \rno_{-1, \Omega_{h,1}} \lno \lambda_h \rno_{1/2, \Gamma^c_{h,1}} \rp,
\end{aligned}
\end{equation*}
and hence
\begin{equation*}
\begin{aligned}
\sum_{i=1,2} B^i_{h,N}\lp \HHH^c_{h,i}\lambda_h^{(i)}, \widehat\HHH^c_{h,i}\lambda_h^{(i)} \rp - Ch\lno \lambda_h \rno^2_{1/2, \Gamma^c_{h,1}} \leq C\sum_{i=1,2} \lno\widetilde f_i \rno_{-1, \Omega_{h,i}} \lno \lambda_h \rno_{1/2, \Gamma^c_{h,1}},
\end{aligned}
\end{equation*}
after seeing that $\widehat \HHH^c_{h,i} \lambda_h^{i} = 0$ on $\Gamma^0_{h,i}$, and applying \eqref{eqn: Neumann Subproblem Boundedness}. We then have that
\begin{equation*}
\begin{aligned}
	&B^i_{h,N}\lp \HHH^c_{h,i} \lambda^{(i)}_h, \widehat \HHH^c_{h,i} \lambda^{(i)}_h \rp\\
	&\quad = 
		B^i_{h,N}\lp \HHH^c_{h,i} \lambda^{(i)}_h, \HHH^c_{h,i} \lambda^{(i)}_h \rp +
		B^i_{h,N}\lp \HHH^c_{h,i} \lambda^{(i)}_h, \lp\widehat\HHH^c_{h,i} - \HHH^c_{h,i}\rp \lambda^{(i)}_h \rp \\\
	&\quad \geq 
		C \lp\gamma_{N,i} - h_i \rp \lno \lambda_h \rno^2_{1/2, \Gamma^c_{h,1}},
\end{aligned}
\end{equation*}
by means of \eqref{eqn: Neumann Subproblem Coercivity}, Lemma \ref{eqn: Perturbation Operator Difference}, \eqref{eqn: Harmonic Lower Bound}, and applying \eqref{eqn: Neumann Subproblem Boundedness}. It then follows that
\begin{equation} \label{eqn: Lambda Bound}
	\lno \lambda_h \rno_{1/2, \Gamma^c_{h,1}} \leq C\sum_{i=1,2} \lno \widetilde f_i \rno_{-1, \Omega_{h,i}}.
\end{equation}

From \eqref{eqn: Solution Decomposition}, we have that
\begin{equation} \label{eqn: U Bound}
	\sum_{i=1,2} \lno u_{h,i} \rno_{1, \Omega_{h,i}} \leq C\lp \lno \lambda_h \rno_{1/2, \Gamma^c_{h,1}} + \sum_{i=1,2} \lno \widetilde f_i \rno_{-1, \Omega_{h,i}} \rp.
\end{equation}
The proof is thus concluded by substituting $\lno \lambda_h \rno_{1/2, \Gamma^c_{h,1}}$ in the above with \eqref{eqn: Lambda Bound}, and subsequently adding \eqref{eqn: Lambda Bound} and \eqref{eqn: Rho Bound} to the resulting bound. 
\end{proof}
\subsection{Error Analysis} \label{section: Error Analysis}
Here, we will present the analysis for the error of \eqref{eqn: MVI}.

First, we state some bounds that will become ubiquitous throughout the error analysis. First, we will denote $L_{i}(\cdot) := -\nabla \cdot \lp \widetilde p_i\lp \xb_i \rp \nabla \lp\cdot\rp \rp$, $\widehat f_i := L_{i} \widetilde u_i$, $u_{p,i} \in V^k_{h,i}$ as $u_{p,i} := \HHH^c_{h,i} \lp\widetilde u_i \circ \etab^c_i \lp \xib^c_i \rp\rp + \GGG^c_{h,i} \widetilde f_i$, and $e_{p,i} := \widetilde u_i - u_{p,i}$. Since $\widehat f_i$ is a bonafide extension of $f_i$, we may apply \cite[Appendix B, Lemma 7]{PE-FEM} to determine that
\begin{equation} \label{eqn: Extension Error}
	\lno \widetilde f_i - \widehat f_i \rno_{-1, \Omega_{h,i}} \leq Ch_i^{2k-2}\lno f_i \rno_{k-1, \Omega_i},
\end{equation}
under the assumption that $\delta_{h,i} \sim \mathcal O\lp h_i^2 \rp$. The error analysis for Dirichlet PE--FEM \cite[Theorem 3]{PE-FEM} imply the following bound
\begin{equation} \label{eqn: PE-FEM Error 1}
	\trino e_{p,i} \trino_{m, \Omega_{h,i}} \leq Ch_i^{k-m+1}\lsn u_i \rsn_{k+1, \Omega_{i}},
\end{equation}
for $m= 1,\ldots, k+1$. Finally, we also have from
\begin{equation*}
	\lno L_i e_{p,i} \rno_{-1, \Omega_{h,i}} = \sup_{\substack {v \in H^1_0(\Omega_{h,i}) \\ \lno v \rno_{1, \Omega_{h,i}} = 1}} \left< \widetilde p_i \nabla e_{p,i}, \nabla v\right>_{\Omega_{h,i}} \leq \lno \widetilde p_i \rno_{C^0\lp\overline{\Omega_{h,i}} \rp} \lno e_{p,i} \rno_{1, \Omega_{h,i}},
\end{equation*}
that the following is bound satisfied
\begin{equation}\label{eqn: PE-FEM Error 2}
	\lno L_{i} e_{p,i} \rno_{-1, \Omega_{h,i}} \leq Ch_i^k \lno u_i \rno_{k+1, \Omega_i}.
\end{equation}

Next, we will decompose $u_{h,i}:= u_{p,i} + z_{h,i}$, where $z_{h,i}\in V^k_{h,i}$ is the discrete error term. Using this decomposition, we are able to write \eqref{eqn: MVI Equivalent} in the following from: seek $\lp z_{h,1}, z_{h,2}, \iota_h, \omega_h  \rp \in \VVV^k_h$ such that
\begin{equation} \label{eqn: Discrete Error Equations}
\begin{aligned}
\begin{aligned}
	B^1_{h,1} \lp z_{h,1}, v_1 \rp - \theta_{h,1} \left< \iota_h, v_1 \right>_{\Gamma^c_{h,1}} &= \left< \kappa_{h,1}, v_1 - \RRR_{h,1} v_1 \right>_{\Gamma^c_{h,1}} \\
	B^2_{h,2} \lp z_{h,2}, v_2 \rp - \theta_{h,2} \left< \iota_h^{(2)}, v_2 \right>_{\Gamma^c_{h,2}} &= \left< \kappa_{h,2}, v_2 - \RRR_{h,2} v_2 \right>_{\Gamma^c_{h,2}} \\
	\sum_{i=1,2} B^i_{h,N} \lp z_{h,i}, \RRR^c_{h,i} \mu_i \rp
		+ \left<\omega_h, \JJJ^c_{1,(2)}\mu_1^{(2)} - \mu_2 \right>_{\Gamma^c_{h,2}} &=
	\MMM_h\lp\mu_1, \mu_2\rp
\end{aligned}
\\
\begin{aligned}
	\quad \forall \lp v_1, v_2, \mu_1, \mu_2 \rp \in \VVV^k_h,
\end{aligned}
\end{aligned}
\end{equation}
where we have defined
\begin{equation*}
\begin{aligned}
	&\iota_h := \lambda_h - \pi^c_1\lp \widetilde u_1 \circ \etab^c_1\lp \xib^c_1 \rp \rp, \quad
	\kappa_{h,i} := \widetilde f_i - \widehat f_i + L_i e_{p,i},\\ 
	&\omega_h := \rho_h - \widetilde p_2\circ \etab^c_2\lp \xib^c_2\rp \left. \Tb^{k-1}_{h,i} \lp\nabla u_{p,i}\rp \right|_{\etab^c_2\lp \xib^c_2 \rp} \cdot \nn_2,
\end{aligned}
\end{equation*}
and 
\begin{equation*}
\begin{aligned}
	\MMM_h\lp \mu_1, \mu_2 \rp &:= \sum_{i=1,2} \lp \left< \widetilde f_i - \widehat f_i, \RRR^c_{h,i}\mu_i \right>_{\Omega_{h,i}} + B^i_{h,N}\lp e_{p,i}, \RRR^c_{h,i}\mu_i \rp\rp \\
	&\quad + \left<\widetilde p_2\circ \etab^c_2\lp \xib^c_2 \rp \left.\Tb^{k-1}_{h,2}\lp \nabla e_{p,2}\rp\right|_{\etab^c_2\lp \xib^c_2 \rp} \cdot \nn_2, \JJJ^c_{1,(2)} \mu_1^{(2)} - \mu_2 \right>_{\Gamma^c_{h,2}} \\
	&\quad +\left<\widetilde p_1\circ \etab^c_1\lp \xib^c_1 \rp \left.\Rb^{k-1}_{h,1}\lp \nabla \widetilde u_1\rp\right|_{\etab^c_1\lp \xib^c_1 \rp} \cdot \nn_1, \mu_1 \right>_{\Gamma^c_{h,1}} \\
	&\quad +\left<\widetilde p_2\circ \etab^c_2\lp \xib^c_2 \rp \left.\Rb^{k-1}_{h,2} \lp\nabla \widetilde u_2\rp\right|_{\etab^c_2\lp \xib^c_2 \rp} \cdot \nn_2, \JJJ^c_{1,(2)}\mu_1^{(2)} \right>_{\Gamma^c_{h,2}}.
\end{aligned}
\end{equation*}
To streamline the exposition of the error analysis, we shall relegate the derivation of \eqref{eqn: Discrete Error Equations} to Appendix \ref{section: Error Equation Derivation}. 
\begin{theorem}
Assume that $\widetilde f_i \in H^{k-1}(\mathbb R^\dimension)$, and that the hypotheses of Theorem \ref{theorem: Stability Bound} hold. Then, the following error bound is satisfied
\begin{equation*}
	\lno \lp \widetilde u - u_{h,1}, \widetilde u - u_{h,2} \rp \rno_{H^1(\Omega_{h,1}) \times H^1(\Omega_{h,2})} \leq C\sum_{i=1,2} \lp h^k\lno u_i \rno_{k+1, \Omega_i} + h^{2k-1} \lno f_i \rno_{k-1, \Omega_i} \rp.
\end{equation*}
\end{theorem}
\begin{proof}
	For notational convenience, we will conform to the notation used in Appendix \ref{sec: Derivation}. 
	We begin our error analysis by seeing that $z_{h,i}$ may be written in the form
\begin{equation*}
	z_{h,i} = \HHH^c_{h,i} \iota_h^{(i)} + \GGG^c_{h,i} \kappa_{h,i}.
\end{equation*}
Then, it follows from \eqref{eqn: Dirichlet Subproblem Stability 1} and Proposition \ref{prop: trace norm equivalence} that
\begin{equation} \label{eqn: Discrete Error Bound 1}
\begin{aligned}
	\lno z_{h,i} \rno_{1, \Omega_{h,i}} &\leq C \lp \lno \iota_h^{(i)} \rno_{1/2, \Gamma^c_{h,i}} + \lno \kappa_{h,i} \rno_{-1, \Omega_{h,i}} \rp \\
	&\leq C \lp \lno \iota_h \rno_{1/2, \Gamma^c_{h,1}} + h_i^{2k-1}\lno f_i \rno_{-1, \Omega_i} + h_i^k\lno u_i\rno_{k+1, \Omega_i} \rp,
\end{aligned}
\end{equation}
after seeing that
\begin{equation} \label{eqn: Kappa Bound}
	\lno \kappa_{h,i} \rno_{-1, \Omega_{h,i}} \leq C\lp h_i^{2k-1}\lno f_i \rno_{-1, \Omega_i} + h_i^k\lno u_i\rno_{k+1, \Omega_i} \rp,
\end{equation}
by virtue of \eqref{eqn: Extension Error} and \eqref{eqn: PE-FEM Error 2}.

We now derive a bound for $\omega_h$. From \eqref{eqn: Separated Equations} we have that
\begin{equation*}
\begin{aligned}
	&\left< \omega_h, \mu_2 \right>_{\Gamma^c_{h,2}}\\
	&\quad= B^2_{h,N}\lp z_{h,2}, \RRR^c_{h,2} \mu_2 \rp - \left<\widetilde f_2 - \widehat f_2, \RRR^c_{h,2}\mu_2 \right>_{\Omega_{h,2}} \\
	&\quad + a_{h,2} \lp e_{p,2}, \RRR^c_{h,2} \mu_2 \rp - \left< E^\prime_{p,2}, \mu_2 \right>_{\Gamma^c_{h,2}} \\
	&\quad \leq C\lp \lno \iota_{h} \rno_{1/2, \Gamma^c_{h,1}} + h_2^{2k-1}\lno f_2 \rno_{k-1, \Omega_i} + h_2^{k-\frac12} \lno u_2 \rno_{k+1, \Omega_2} \rp \lno \mu_2 \rno_{1/2, \Gamma^c_{h,2}},
\end{aligned}
\end{equation*}
after applying \eqref{eqn: Neumann Subproblem Boundedness}, \eqref{eqn: Flux Trace Error 1} of Lemma \ref{lemma: Trace Derivative Error}, \eqref{eqn: Extension Error}, and \eqref{eqn: Discrete Error Bound 1}. It then follows from the definition of the dual norm that
\begin{equation} \label{eqn: Omega Bound}
	\lno \omega_h \rno_{-1/2, \Gamma^c_{h,2}} \leq  C\lp \lno \iota_{h} \rno_{1/2, \Gamma^c_{h,1}} + h_2^{2k-1}\lno f_2 \rno_{k-1, \Omega_i} + h_2^{k-\frac12} \lno u_2 \rno_{k+1, \Omega_2} \rp.
\end{equation}

Next, we derive a bound for $\iota_h$. Let us set $\mu_1 = \iota_h$, $\mu_2 = \pi^c_2 \iota_h$, $\RRR^c_{h,i}\lp \cdot \rp = \widehat \HHH^c_{h,i} \lp \cdot \rp$ in \eqref{eqn: Discrete Error Equations}. 
It follows that the third equation of \eqref{eqn: Discrete Error Equations} can be written as
\begin{equation}\label{eqn: Third Equation Decomposed}
\begin{aligned}
	\sum_{i=1,2} B^i_{h,N} \lp \HHH^c_{h,i} \iota^{(i)}_h, \widehat \HHH^c_{h,i} \iota^{(i)}_h \rp + \left< \omega_h, \JJJ^c_{1, (2)} \iota_h^{(2)} - \pi^c_2\iota_h^{(2)} \right>_{\Gamma^c_{h,2}}\\
	= \MMM_h\lp \iota_h^{(2)}, \pi^c_2\iota_h^{(2)} \rp - \sum_{i=1,2} B^i_{h,N} \lp \GGG^c_{h,i} \kappa_{h,i}, \widehat \HHH^c_{h,i} \iota^{(i)}_h  \rp.
\end{aligned}
\end{equation}
First, we have that
\begin{equation} \label{eqn: BN Bound}
	B^i_{h,N} \lp \GGG^h_{h,i} \kappa_i, \widehat \HHH^c_{h,i} \iota_h^{(i)} \rp \leq C\lp h_i^{2k-1} \lno \widetilde f_i \rno_{k-1, \Omega_i} + h_i^{k} \lno u_i \rno_{k+1, \Omega_i} \rp,
\end{equation}
by \eqref{eqn: Neumann Subproblem Boundedness 2}, \eqref{eqn: Dirichlet Subproblem Stability 1}, and \eqref{eqn: Kappa Bound} . Then, utilizing \eqref{eqn: Jacobian Bound}, Proposition \ref{prop: trace norm equivalence} and the projection theorem, implies that
\begin{equation} \label{eqn: Perturbation Bound}
\begin{aligned}
	&\left< \omega_h, \JJJ^c_{1,(2)} \iota_h^{(2)} - \pi^c_2 \iota_h^{(2)} \right>_{\Gamma^c_{h,2}}\\
	&\quad = \left< \rho_h - F_{p,2}, \JJJ^c_{1,(2)} \iota_h^{(2)} - \pi^c_2 \iota_h^{(2)} \right>_{\Gamma^c_{h,2}} \\
	&\quad=\left< \rho_h - \pi^c_2 F_{p,2}, \lp \JJJ^c_{1,(2)}-1\rp\iota_h^{(2)}\right>_{\Gamma^c_{h,2}} + \left< \pi^c_2 F_{p,2} - F_{p,2}, \JJJ^c_{1,(2)}\iota_h^{(2)} \right>_{\Gamma^c_{h,2}} \\
	&\quad=\left<\omega_h, \lp\JJJ^c_{1,(2)} - 1\rp \iota_h^{(2)} \right>_{\Gamma^c_{h,2}} + \left< \pi^c_2 F_{p,2} - F_{p,2}, \iota_h^{(2)} - \pi^c_2 \iota_h^{(2)} \right>_{\Gamma^c_{h,2}} \\
	&\quad\leq C\lp h \lno \iota_h \rno_{1/2, \Gamma^c_{h,1}} + h^{2k} \lno f_2 \rno_{k-1, \Omega_2} + h^k_2 \lno u_2 \rno_{k+1, \Omega_2} \rp \lno \iota_h \rno_{1/2, \Gamma^c_{h,1}},
\end{aligned}
\end{equation}
since
 $$\lno \iota_h^{(2)} - \pi^c_2 \iota_h^{(2)} \rno_{0, \Gamma^c_{h,2}} \leq Ch_2^\frac12 \lno \iota_h \rno_{1/2, \Gamma^c_{h,1}},$$
from standard approximation arguments, and 
\begin{equation*}
\begin{aligned}
	\lno \pi^c_2 F_{p,2} - F_{p,2} \rno_{0, \Gamma^c_{h,2}} &=
		\lno \pi^c_2 F_{p,2} - F_2 \rno_{0, \Gamma^c_{h,2}} 
		+ \lno E_{p,2} \rno_{0, \Gamma^c_{h,2}}\\
		&= \lno \pi^c_2 E_{p,2} \rno_{0, \Gamma^c_{h,2}} + \lno \pi^c_2 F_2 - F_2 \rno_{0, \Gamma^c_{h,2}} + \lno E_{p,2} \rno_{0, \Gamma^c_{h,2}} \\
		&\leq C h^{k-\frac12}\lno u_2 \rno_{k+1, \Omega_2},
\end{aligned}
\end{equation*}
after applying \eqref{eqn: Flux Trace Error 2} of Lemma \ref{lemma: Trace Derivative Error}, \cite[Appendix A, Lemma 2]{PE-FEM} with $\delta_{h,2} \sim \mathcal O\lp h_2^2 \rp$, and seeing that $F_2 \in H^{k-\frac12}\lp \Gamma^c_{h,2} \rp$. 
Next, we proceed to bound $\MMM_h\lp \iota_h, \pi^c_2 \iota_h^{(2)} \rp$; let us denote
\begin{equation*}
	\widehat E_{p,2} := \widetilde p_2\circ \etab^c_2 \lp \xib^c_2 \rp \left.\Tb^{k-1}_{h,2} \lp \nabla e_{p,2} \rp\right|_{\etab^c_2 \lp \xib^c_2 \rp} \cdot \nn_2,
\end{equation*}
then the second term in $\MMM_h\lp \iota_h, \pi^c_2 \iota_h^{(2)}\rp$ can be bounded as follows
\begin{equation*}
\begin{aligned}
	&\left< \widehat E_{p,2}, \JJJ^c_{1,(2)}\iota_h^{(2)} - \pi^c_2 \iota_h^{(2)} \right>_{\Gamma^c_{h,2}}\\
	&\quad\leq \lno \widehat E_{p,2} \rno_{0, \Gamma^c_{h,2}} \lp \lno \iota_h^{(2)} - \pi^c_2 \iota_h^{(2)} \rno_{0, \Gamma^c_{h,2}} + \lno \lp \JJJ^c_{1,(2)} - 1 \rp \iota_h^{(2)} \rno_{1/2, \Gamma^c_{h,2}} \rp \\
	&\quad\leq C h^{k} \lno u_2\rno_{k+1, \Omega_2}\lno \iota_h \rno_{1/2, \Gamma^c_{h,1}},
\end{aligned}
\end{equation*}
after applying \eqref{eqn: Flux Trace Error 1} of Lemma \ref{lemma: Trace Derivative Error}, \eqref{eqn: Jacobian Bound}, and a standard approximation argument. 
Subsequently, we have that
\begin{equation*} 
	B^i_{h,N}\lp e_{p,i}, \widehat \HHH^c_{h,i} \iota_h^{(i)} \rp \leq 
		Ch_i^k \lno u_i \rno_{k+1, \Omega_i},
\end{equation*}
by virtue of Lemma \ref{lemma: Tau Bounds} and \eqref{eqn: PE-FEM Error 1}. It then follows easily from \cite[Appendix A, Lemma 2]{PE-FEM}, \eqref{eqn: Extension Error}, \eqref{eqn: Neumann Subproblem Boundedness}, and \eqref{eqn: PE-FEM Error 1} that
\begin{equation} \label{eqn: M Bound}
	\MMM_h\lp \iota_h, \pi^c_2 \iota_h^{(2)} \rp \leq C\sum_{i=1,2} \lp h_i^{2k-1} \lno f_i \rno_{k-1, \Omega_i} + h^{k} \lno u_i \rno_{k+1, \Omega_i} \rp\lno \iota_h \rno_{1/2, \Gamma^c_{h,1}}. 
\end{equation}	
Finally, we have that 
\begin{equation*}
\begin{aligned}
B^i_{h,N} \lp \HHH^c_{h,i} \iota^{(i)}_h, \widehat \HHH^c_{h,i} \iota^{(i)}_h \rp &= B^i_{h,N} \lp \HHH^c_{h,i} \iota^{(i)}_h, \HHH^c_{h,i} \iota^{(i)}_h \rp\\ 
&\quad+ B^i_{h,N} \lp \HHH^c_{h,i} \iota^{(i)}_h, \lp \HHH^c_{h,i} - \widehat \HHH^c_{h,i}\rp \iota^{(i)}_h \rp \\
&\geq C(\gamma_{N,i} - h_i)\lno \iota_{h} \rno^2_{1/2, \Gamma^c_{h,1}},
\end{aligned} 
\end{equation*}
from Lemma \ref{eqn: Perturbation Operator Difference}, \eqref{eqn: Neumann Subproblem Coercivity}, \eqref{eqn: Harmonic Lower Bound}, and \eqref{eqn: Neumann Subproblem Boundedness}. The above bound together with \eqref{eqn: Perturbation Bound}, \eqref{eqn: M Bound}, and \eqref{eqn: BN Bound} yields the following
\begin{equation}\label{eqn: Iota Bound}
	\lno \iota_h \rno_{1/2, \Gamma^c_{h,i}} \leq C \sum_{i=1,2} \lp h^k \lno u_i \rno_{k+1, \Omega_i} + h^{2k-1} \lno f_i \rno_{k-1, \Omega_i} \rp.
\end{equation}

The proof is therefore concluded by recalling \eqref{eqn: Discrete Error Bound 1}, 
\begin{equation*}
	\lno \widetilde u_i - u_{h,i} \rno_{1, \Omega_{h,i}} \leq \lno e_{p,i} \rno_{1, \Omega_{h,i}} + \lno z_{h,i} \rno_{1, \Omega_{h,i}},
\end{equation*}
and applying \eqref{eqn: PE-FEM Error 1}, \eqref{eqn: Discrete Error Bound 1}, and \eqref{eqn: Iota Bound}.
\end{proof}
\section{Numerical Illustration} \label{section: Numerical Illustration}
In this section, we present an illustrative numerical example for this coupling approach. We consider the interface problem where $\Omega_1$ is a disk of radius $\frac14$ and $\Omega_2$ is an annulus with an inner radius of $\frac14$ and an outer radius of $\frac12$.  $\Omega_{h,1}$ and $\Omega_{h,2}$ are discretized so that each element on $\Gamma^c_{h,1}$ and $\Gamma^c_{h,2}$ respectively have the same length. Additionally, we fix the ratio of elements between $\Gamma^c_{h,1}$ and $\Gamma^c_{h,2}$ to be $1:1$, $1:2$, and $2:1$. The manufactured solution we consider is the following
\begin{equation*}
\left\{
\begin{aligned}
u_1 &= e^{-5(x^2 + y^2)} \\
u_2 &= \frac{e^{-5(x^2 + y^2)} + e^{-\frac{5}{16}}}{2},
\end{aligned}
\right.
\end{equation*}
This solution corresponds to \eqref{eqn: Elliptic Interface Problem} with $p_1 = 1$, $p_2 = 2$, $f_1 = f_2 = -100(x^2 + y^2) + 20e^{-5(x^2 + y^2)}$. The numerical solutions are computed using quadratic, cubic, and quartic Lagrange elements on triangles.  We present the convergence histories in the following tables.
\begin{table}[h!]
\begin{tabular}{|c|c|c|c|c|c|c|}
\hline
\multicolumn{7}{|c|}{ $1:1$ Ratio}\\
\hline
& \multicolumn{2}{|c|}{Quadratic} & \multicolumn{2}{|c|}{Cubic} & \multicolumn{2}{|c|}{Quartic}\\
\hline
$h$ & $L^2$--Error & $H^1$--Error & $L^2$--Error & $H^1$--Error & $L^2$--Error & $H^1$--Error \\
\hline
0.3827 & 4.7000E-3 & 5.7280E-2 & 4.9120E-4 & 7.33919E-3 & 1.7986E-4& 1.4585E-3 \\
0.2199 & 6.4000E-4 & 1.5427E-2 & 1.7007E-5 & 1.04039E-3 & 5.1535E-6 & 9.0318E-5 \\
0.1200 & 7.6813E-5 & 3.6798E-3 & 1.0265E-6 & 1.29739E-4 & 1.2813E-7 & 5.3379E-6 \\
0.0700 & 1.0906E-5 & 1.0265E-3 & 9.1517E-8 & 2.03711E-5 & 5.4741E-9 & 4.0525E-7 \\
0.0349 & 1.3666E-6 & 2.5706E-4 & 5.6586E-9 & 2.55939E-6 & &\\
\hline
Rate &  3.424 & 2.274 & 4.699 & 3.340 &  6.120 & 4.803 \\
\hline 
\end{tabular}
\begin{tabular}{|c|c|c|c|c|c|c|}
\hline
\multicolumn{7}{|c|}{ $1:2$ Ratio}\\
\hline
& \multicolumn{2}{|c|}{Quadratic} & \multicolumn{2}{|c|}{Cubic} & \multicolumn{2}{|c|}{Quartic}\\
\hline
$h$ & $L^2$--Error & $H^1$--Error & $L^2$--Error & $H^1$--Error & $L^2$--Error & $H^1$--Error \\
\hline
0.2635 & 2.7383E-3 & 2.8487E-2 & 3.0674E-4 & 3.5639E-3 & 8.4549E-5 & 5.4429E-4 \\
0.1458 & 3.6003E-4 & 8.0675E-3 & 6.5463E-6 & 4.5505E-4 & 2.5055E-6 & 2.7702E-5 \\
0.0789 & 4.0748E-5 & 1.9102E-3 & 3.6973E-7 & 4.7691E-5 & 4.3316E-8 & 1.3256E-6 \\ 
0.0413 & 4.8422E-6 & 4.7847E-4 & 4.0131E-8 & 6.3799E-6 & 4.1001E-10& 8.1199E-8 \\
0.0201 & 4.7705E-7 & 1.1875E-5 & 2.2039E-9 & 8.3134E-7 & & \\
\hline
Rate &  3.370 & 2.148 & 4.472 & 3.268 &  6.608 & 4.769 \\
\hline 
\end{tabular}
\begin{tabular}{|c|c|c|c|c|c|c|}
\hline
\multicolumn{7}{|c|}{ $2:1$ Ratio}\\
\hline
& \multicolumn{2}{|c|}{Quadratic} & \multicolumn{2}{|c|}{Cubic} & \multicolumn{2}{|c|}{Quartic}\\
\hline
$h$ & $L^2$--Error & $H^1$--Error & $L^2$--Error & $H^1$--Error & $L^2$--Error & $H^1$--Error \\
\hline
0.3827 & 3.5786E-3 & 4.9183E-2 & 2.8360E-4 & 6.1253E-3 & 1.0073E-4 & 1.2116E-3 \\
0.2190 & 4.7034E-4 & 1.3388E-2 & 1.6031E-5 & 9.4284E-4 & 2.5941E-6 & 8.3794E-5 \\
0.1200 & 5.6706E-5 & 3.2844E-3 & 1.0048E-6 & 1.2237E-4 & 6.6622E-8 & 5.1549E-6 \\ 
0.0670 & 8.5156E-6 & 9.4006E-4 & 9.0032E-8 & 1.9515E-5 & 3.3644E-9 &3.9700E-7 \\
0.0348 & 1.0876E-6 & 2.3633E-4 & 5.5777E-9 & 2.4412E-6 & &  \\
\hline
Rate &  3.398 & 2.242 & 4.515 & 3.286 &  6.066 &  4.709 \\
\hline 
\end{tabular}
\caption{Convergence histories of numerical tests. We observe optimal convergence in the broken $L^2$ and $H^1$ norms.}
\end{table}
\section{Concluding Remarks} \label{section: Concluding Remarks}
In this paper, we have presented a new method, based on PE--FEM, for coupling numerical solutions together on geometrically nonmatching discrete interface approximations. Stability and optimal $H^1$ error bounds was proven, and the numerical illustration confirms our theoretical findings. Additionally, the numerical results imply that our method is optimally convergent in $L^2$--as well. In future works, we will apply and analyze generalizations of this coupling approach to other interface phenomena, such as groundwater flows and fluid--structure interaction. 
\appendix
\section{Technical Lemmas}
\begin{lemma} \label{lemma: Trace Derivative Error}
Let $e_{p,i} \in H^1(\Omega_{h,i}), H^{k+1}\lp \KKK_{h,i} \rp$ be defined as in Section \ref{section: Error Analysis}, then the following bounds are satisfied
\begin{equation} \label{eqn: Flux Trace Error 1}
	\lno \widetilde p_i \lp \xib^c_i \rp \nabla e_{p,i} \cdot \nn_{h,i} \rno_{0, \Gamma^c_{h,i}} \leq C\lno \widetilde p_i \rno_{C^0\lp \overline{\Omega_{h,i}} \rp}h_i^{k-\frac12} \lno u_i \rno_{k+1, \Omega_i}
\end{equation}
\begin{equation} \label{eqn: Flux Trace Error 2}
	\lno \widetilde p_i \circ \etab^c_i\lp \xib^c_i \rp \left.\Tb^{k-1}_{h,i} \lp\nabla e_{p,i}\rp\right|_{\etab^c_i\lp \xib^c_i \rp} \cdot \nn_{i} \rno_{0, \Gamma^c_{h,i}} \leq Ch_i^{k-\frac12} \lno u_i \rno_{k+1, \Omega_i}.
\end{equation}
\end{lemma}
\begin{proof}
First, we have that
\begin{equation*}
\begin{aligned}
	\lno \widetilde p_i \lp \xib^c_i \rp \nabla e_{p_i} \cdot \nn_{h,i} \rno_{0, \Gamma^c_{h,i}} &\leq C \lno \widetilde p_i \rno_{C^0\lp \overline{\Omega_{h,i}} \rp} \lno \nabla e_{p,i} \rno_{0, \Gamma^c_{h,i}} \\
	&\leq C\lno \widetilde p_i \rno_{C^0\lp \overline{\Omega_{h,i}} \rp} \lno e_{p,i} \rno_{1, \Omega_{h,i}} \trino e_{p,i} \trino_{2, \Omega_{h,i}} \\
	&\leq Ch^{k-1/2}_i\lno \widetilde p_i \rno_{C^0\lp \overline{\Omega_{h,i}} \rp} \lno u_u \rno_{k+1, \Omega_i},
\end{aligned}
\end{equation*}
by virtue of \eqref{eqn: PE-FEM Error 1}, thus we have \eqref{eqn: Flux Trace Error 1}.

Next, we have that
\begin{equation*}
\lno \widetilde p_i \circ \etab^c_i\lp \xib^c_i \rp \left.\Tb^{k-1}_{h,i} \lp\nabla e_{p,i}\rp\right|_{\etab^c_i\lp \xib^c_i \rp} \cdot \nn_{i} \rno_{0, \Gamma^c_{h,i}} \leq \lno\widetilde p_i  \rno_{C^0\lp\overline{\Omega_{h,i}}\rp} \lno \left.\Tb^{k-1}_{h,i} \lp\nabla e_{p,i}\rp\right|_{\etab^c_i\lp \xib^c_i\rp}\rno_{0, \Gamma^c_{h,i}}.
\end{equation*}
It then follows that
\begin{equation*}
\begin{aligned}
	\lno \left.\Tb^{k-1}_{h,i} \lp \nabla e_{p,i} \rp \right|_{\etab^c_i\lp\xib^c_i\rp} \rno_{0, \Gamma^c_{h,i}} &\leq 	\lno \left.\Tb^{k-1}_{h,i} \lp \nabla e_{p,i} \rp \right|_{\etab^c_i\lp\xib^c_i\rp} - \nabla e_{p.i}\rno_{0, \Gamma^c_{h,i}} + \lno e_{p,i} \rno_{0, \Gamma^c_{h,i}}\\
	&\leq C \lp h_i \trino e_{p,i} \trino_{2, \Omega_{h,i}} + h^{k-1}_i \trino e_{p,i} \trino_{k+1, \Omega_{h,i}} \rp + \lno \nabla e_{p,i} \rno_{0, \Gamma^c_{h,i}} \\
	&\leq C h_i^{k-\frac12} \lno u_i \rno_{k+1, \Omega_i},
\end{aligned}
\end{equation*}
by virtue of \cite[Appendix A, Lemma 3]{PE-FEM}. Thus we have \eqref{eqn: Flux Trace Error 2}. 
\end{proof}
\begin{lemma} \label{lemma: Tau Bounds}
	Let $v\in L^2(\Omega_{h,i}), H^{k+1}(\KKK_{h,i})$ and assume that $\delta_{h,i} \sim O(h_i^2)$, then for $k=1,2,\ldots$
\begin{equation} \label{eqn: Tau Bound 1}
\begin{aligned}
	\tau_i\lp v,\mu \rp \leq C \lno \widetilde p_i \rno_{C^1\lp \overline{\Omega_i \cap \Omega_{h,i}} \rp} \lp h_i \trino v \trino_{2, \Omega_{h,i}} + h_i^{2k-1} \trino v \trino_{k+1, \Omega_{h,i}} \rp \lno \mu \rno_{1/2, \Gamma^c_{h,i}}\\
	 \quad \forall \mu \in H^{1/2}(\Gamma^c_{h,i}).
\end{aligned}
\end{equation}
\end{lemma}
\begin{proof}
	The definition of $\tau_i\lp\cdot, \cdot \rp$, see \eqref{eqn: Tau Definition}, and a triangle inequality, we have that
\begin{equation*}
\begin{aligned}
	\tau_i\lp v, \mu \rp &\leq \lno \mu \rno_{1/2, \Gamma^c_{h,i}}
		\bigg(
		\lno \widetilde p_i\circ\etab^c_i\lp \xib^c_i \rp \lp\left.\Tb^{k-1}_{h,i} \nabla v\right|_{\etab^c_i\lp\xib^c_i\rp} - \nabla v\lp\xib^c_i\rp\rp \cdot \nn_i \rno_{0, \Gamma^c_{h,i}}\\
		&\qquad+ \lno \lp\widetilde p_i\circ\etab^c_i\lp\xib^c_i\rp - \widetilde p_i\lp\xib^c_i\rp \rp  \nabla v\lp \xib^c_i \rp\cdot \nn_i\rno_{0, \Gamma^c_{h,i}} \\
		&\qquad+\lno \widetilde p_i\lp\xib^c_i\rp \nabla v\lp\xib^c_i\rp\cdot\lp \nn_i - \nn_{h,i}\rp \rno_{0, \Gamma^c_{h,i}}
		 \bigg)
\end{aligned}
\end{equation*}
From \cite[Appendix A, Lemma 3]{PE-FEM} we have that
\begin{equation*}
\begin{aligned}
	&\lno \widetilde p_i\etab^c_i\lp \xib^c_i \rp \lp\left.\Tb^{k-1}_{h,i} \nabla v\right|_{\etab^c_i\lp\xib^c_i\rp} - \nabla v\lp\xib^c_i\rp\rp \cdot \nn_i \rno_{0, \Gamma^c_{h,i}} \\
	&\qquad \leq \lno \widetilde p_i \rno_{C^0(\overline{\Omega_i \cup \Omega_{h,i}})}
		\lp h_i \trino v \trino_{2, \Omega_{h,i}} + h_i^{2k-1}\trino v \trino_{k+1, \Omega_{h,i}} \rp.
\end{aligned}
\end{equation*}
We then also have that
\begin{equation*}
\begin{aligned}
\lno \lp\widetilde p_i\etab^c_i\lp\xib^c_i\rp - \widetilde p_i\lp\xib^c_i\rp \rp  \nabla v\lp \xib^c_i \rp\cdot \nn_i\rno_{0, \Gamma^c_{h,i}} &\leq 
	Ch_i^2\lno \widetilde p_i \rno_{C^1(\overline{\Omega_i \cup \Omega_{h,i}})}\lno \nabla v \rno_{0, \Gamma^c_{h,i}} \\
	&\leq Ch_i^2\lno \widetilde p_i \rno_{C^1(\overline{\Omega_i \cup \Omega_{h,i}})}\trino v \trino_{2, \Omega_{h,i}},
\end{aligned}
\end{equation*}
after applying Taylor's theorem for continuous functions and the trace inequality. Finally, since $\lsn\nn_i - \nn_{h,i}\rsn \sim \mathcal O(h_i)$, we have that
\begin{equation*}
\lno \widetilde p_i\lp\xib^c_i\rp \nabla v\lp\xib^c_i\rp\cdot\lp \nn_i - \nn_{h,i}\rp \rno_{0, \Gamma^c_{h,i}} \leq Ch_i \lno \widetilde p_i \rno_{C^0\lp\overline{\Omega_i \cup \Omega_{h,i}} \rp} \trino v \trino_{2, \Omega_{h,i}}.
\end{equation*}
Hence \eqref{eqn: Tau Bound 1} is established.
\end{proof}
\section{Derivation of \eqref{eqn: Discrete Error Equations}} \label{section: Error Equation Derivation} \label{sec: Derivation}
We begin our derivation by seeing that
\begin{equation*}
\begin{aligned}
	B^i_{h,D}\lp u_{h,i}, v_i \rp &= B^i_{h,D} \lp z_{h,i}, v_i \rp + \theta_{h,i} \left<\left. T^k_{h,i} u_{p,i} \right|_{\etab^c_i\lp \xib^c_i \rp}, v_i \right>_{\Gamma^c_{h,i}}\\
	 &+ \theta_{h,i} \left<\left. T^k_{h,i} u_{p,i} \right|_{\etab^0_i\lp \xib^0_i \rp}, v_i \right>_{\Gamma^0_{h,i}} + \left<L_i u_{p,i}, v_i \right>_{\Omega_{h,i}} \quad \forall v_i \in V^k_{h,i},
\end{aligned}
\end{equation*}
by the definition of $B^i_{h,D}\lp \cdot, \cdot \rp$ given in \eqref{eqn: Dirichlet Bilinear Form} and an application of Green's identity. We then have that
\begin{equation*}
	\left<\widetilde f_i, v_i - \RRR_{h,i} v_i \right>_{\Omega_{h,i}} = 
	\left<\widetilde f_i - \widehat f_i + L_i \widetilde u_{i}, v_i - \RRR_{h,i} v_i \right>_{\Omega_{h,i}},
\end{equation*}
by definition of $\widehat f_i$. It then follows that
\begin{equation*}
	B^i_{h,D}\lp z_{h,i}, v_i \rp - \theta_{h,i} \left< \iota_{h}^{(i)}, v_i \right>_{\Gamma^c_{h,i}} = 
	\left< \kappa_{h,i}, v_i - \RRR_{h,i} v_i \right>_{\Omega_{h,i}} \quad \forall v_i \in V^k_{h,i},
\end{equation*}
since by the definition of $u_{p,i}$, we have that
\begin{equation*}
	\left< \left. T^k_{h,i}  u_{p,i}\right|_{\etab^c_i\lp \xib^c_i \rp} - \widetilde u_i \circ \etab^c_i\lp \xib^c_i \rp, \mu_i\right>_{\Gamma^c_{h,i}}  = 0 \quad \forall \mu_i \in W^{c,k}_{h,i}
\end{equation*}
and
\begin{equation*}
	\left< \left. T^k_{h,i}  u_{p,i}\right|_{\etab^0_i\lp \xib^0_i \rp}, \mu_i\right>_{\Gamma^0_{h,i}}  = 0 \quad \forall \mu_i \in W^{0,k}_{h,i}.
\end{equation*}
Hence, the first two equations of \eqref{eqn: Discrete Error Equations} are derived. 

Next, we will derive the third equation of \eqref{eqn: Discrete Error Equations}. For notational convenience, we will denote
\begin{equation*}
\begin{aligned}
	F_i &:= \widetilde p_i \circ \etab^c_i \lp \xib^c_i\rp \lp\nabla \widetilde u_i\circ \etab^c_i \lp \xib^c_i \rp \rp\cdot \nn_i, \\
	F_{p,i} &:= \widetilde p_i \circ \etab^c_i \lp \xib^c_i \rp \left.\Tb^{k-1}_{h,i} \lp \nabla u_{p,i}\rp\right|_{\etab^c_i \lp \xib^c_i \rp} \cdot \nn_i,\\
	E^\prime_{p,i} &:= \widetilde p_i\lp \xib^c_i \rp \nabla e_{p,i} \cdot \nn_{h,i},
\end{aligned}
\end{equation*}
and
\begin{equation*}
	E_{p,i} := \widetilde p_i \circ \etab^c_i \lp \xib^c_i \rp \lp\left.\Tb^{k-1}_{h,i} \lp \nabla e_{p,i}\rp\right|_{\etab^c_i \lp \xib^c_i \rp} + \left.\Rb^{k-1}_{h,i} \lp \nabla \widetilde u_i \rp\right|_{\etab^c_i \lp \xib^c_i \rp}\rp \cdot \nn_i.
\end{equation*}
Then, from a similar set of steps used in the previous paragraph, we have that the last two equations of \eqref{eqn: MVI} can be written as
\begin{equation} \label{eqn: Separated Equations}
\begin{aligned}
\begin{aligned}
	&B^1_{h,N} \lp z_{h,1}, \RRR^c_{h,1} \mu_1 \rp + \left< \rho_h^{(1)} + F_{p,1}, \mu_1 \right>_{\Gamma^c_{h,1}}\\
	 &\quad= \left< \widetilde f_1 - \widehat f_1, \RRR^c_{h,1} \mu_1 \right>_{\Omega_{h,1}} + a_{h,1}\lp e_{p,1}, \RRR^c_{h,1} \mu_1 \rp - \left< E^\prime_{p,1}, \mu_1 \right>_{\Gamma_{h,1}}
\end{aligned}
\\
\begin{aligned}
	&B^2_{h,N} \lp z_{h,2}, \RRR^c_{h,2} \mu_2 \rp - \left< \rho_h - F_{p,2}, \mu_2 \right>_{\Gamma^c_{h,2}}\\
	 &\quad= \left< \widetilde f_2 - \widehat f_2, \RRR^c_{h,2} \mu_2 \right>_{\Omega_{h,2}} + a_{h,2}\lp e_{p,2}, \RRR^c_{h,2} \mu_2 \rp - \left< E^\prime_{p,2}, \mu_2 \right>_{\Gamma_{h,2}}
\end{aligned}
\\
\begin{aligned}
	\forall \lp\mu_1, \mu_2\rp \in \WWW^k_h,
\end{aligned}
\end{aligned}
\end{equation}
We now see that
\begin{equation} \label{eqn: Trace Duality Term}
\begin{aligned}
&\left< \rho_h^{(1)} + F_{p,1}, \mu_1 \right>_{\Gamma^c_{h,1}} - \left< \rho_h - F_{p,2}, \mu_2 \right>_{\Gamma^c_{h,2}} \\
&\quad = \left< \rho_h + F_{p,1}^{(2)}, \JJJ^c_{1,(2)}\mu_1^{(2)} \right>_{\Gamma^c_{h,2}} - \left< \rho_h - F_{p,2}, \mu_2 \right>_{\Gamma^c_{h,2}} \\
&\quad = \left< F_{p,1}^{(2)} + F_{p,2}, \JJJ^c_{1,(2)} \mu_1^{(2)} \right>_{\Gamma^c_{h,2}} + \left< \rho_h - F_{p,2}, \JJJ^c_{1,(2)}\mu_1^{(2)} - \mu_2 \right>_{\Gamma^c_{h,2}} \\
&\quad = -\left< E_{p,1}^{(2)} + E_{p,2}, \JJJ^c_{1, (2)} \mu_1^{(2)}\right>_{\Gamma^c_{h,2}} + \left< \rho_h - F_{p,2}, \JJJ^c_{1,(2)}\mu_1^{(2)} - \mu_2 \right>_{\Gamma^c_{h,2}}\\
&\quad = -\left< E_{p,1}, \mu_1 \right>_{\Gamma^c_{h,1}} - \left< E_{p,2}, \mu_2 \right>_{\Gamma^c_{h,2}} - \left< E_{p,2}+ \rho_h - F_{p,2}, \JJJ^c_{1, (2)} \mu_1^{(2)} - \mu_2 \right>_{\Gamma^c_{h,2}}.
\end{aligned}
\end{equation}
where we have added the productive zero in the third equality, i.e.,
\begin{equation*}
	-p_1\frac{\partial u_1}{\partial \nn_1} - p_2\frac{\partial u_2}{\nn_2}
	=-\sum_{i=1,2} F_i = 0.
\end{equation*}
Adding together the equations \eqref{eqn: Separated Equations}, and applying the change of variables formula on the trace duality term over $\Gamma^c_{h,1}$, substituting \eqref{eqn: Trace Duality Term}, and seeing that
\begin{equation*}
	\left< E_{p,i} - E^\prime_{p,i}, \mu_i \right>_{\Gamma^c_{h,i}} = \tau_i\lp e_{p,i}, \mu_i \rp
\end{equation*}
 yields the third equation in \eqref{eqn: Discrete Error Equations}.
%
%
\nocite{*}
\bibliographystyle{plain}
\bibliography{bibliography.bib}
\end{document}